\documentclass[a4paper,10pt]{article}

\usepackage[english]{babel}
\usepackage{amsmath,amsthm,amssymb}
\usepackage{url}
\usepackage{color}
\usepackage{authblk}
\usepackage[colorlinks=true]{hyperref}

\newtheorem{theorem}{Theorem}[section]
\newtheorem{lemma}[theorem]{Lemma}
\newtheorem{proposition}[theorem]{Proposition}
\newtheorem{corollary}[theorem]{Corollary}
\newtheorem{definition}{Definition}
\newtheorem{remark}{Remark}

\numberwithin{equation}{section}
\numberwithin{figure}{section}
\numberwithin{table}{section}

\title{Optimal convergence of adaptive FEM for eigenvalue clusters in mixed form}
\author[1]{D.\ Boffi\thanks{daniele.boffi@unipv.it}}
\author[2]{D.\ Gallistl\thanks{gallistl@ins.uni-bonn.de}}
\author[1]{F.\ Gardini\thanks{francesca.gardini@unipv.it}}
\author[3]{L.\ Gastaldi\thanks{lucia.gastaldi@unibs.it}}
\affil[1]{Dipartimento di Matematica ``F. Casorati'', University of Pavia,
Italy}
\affil[2]{Institut f\"ur Numerische Simulation, University of Bonn, Germany}
\affil[3]{DICATAM, University of Brescia, Italy}

\definecolor{purple}{rgb}{0.5, 0.0, 0.5}
\definecolor{violet}{rgb}{0.56, 0.0, 1.0}
\definecolor{darkspringgreen}{rgb}{0.09, 0.45, 0.27}
\definecolor{grey}{rgb}{0.6, 0.6, 0.6}

\newcommand{\A}{\mathsf{A}}
\newcommand{\B}{\mathsf{B}}
\newcommand{\N}{\mathsf{N}}
\renewcommand{\H}{\mathcal{H}}
\newcommand{\eigl}{E(\lambda)}
\newcommand{\fortin}{\Pi^F_h}
\newcommand{\fortinH}{\Pi^F_H}
\newcommand{\kernel}{\mathbb{K}_h}
\newcommand{\RE}{\mathbb{R}}
\newcommand{\ssigma}{\boldsymbol{\sigma}}
\newcommand{\ttau}{\boldsymbol{\tau}}
\newcommand{\x}{\mathbf{x}}
\newcommand{\ddiv}{\operatorname{div}}
\newcommand{\card}{\operatorname{card}}
\newcommand{\curl}{\operatorname{curl}}
\newcommand{\ccurl}{\operatorname{\mathbf{curl}}}
\newcommand{\T}{\mathcal{T}}
\newcommand{\HH}{\H}

\newcommand{\G}{\textbf{\textsf{G}}}
\newcommand{\Tlh}{T^\lambda_h}
\newcommand{\Tll}{T^{\lambda_j}_\ell}

\newcommand{\Thl}{T_h^\lambda}

\newcommand{\Crel}{C_{\operatorname{rel}}}
\newcommand{\Cdrel}{C_{\operatorname{drel}}}
\newcommand{\Cqo}{C_{\operatorname{qo}}}

\newenvironment{sistemato?}
{\bgroup\color{darkspringgreen}\textbf{Tentativo di dimostrazione.}}
{\egroup}

\begin{document}

\maketitle

\begin{abstract}
It is shown that the $h$-adaptive mixed finite element method
for the discretization of eigenvalue clusters of the Laplace operator  
produces optimal convergence rates in terms of nonlinear approximation
classes.
The results are valid for the
typical mixed spaces of Raviart--Thomas or Brezzi--Douglas--Marini type
with arbitrary fixed polynomial degree in two and three space dimensions.

\end{abstract}

\section{Introduction}

The study of optimal convergence rates for adaptive finite element schemes has
been carried on by several researchers during the last decades in the case of
source problems 
(see, e.g.,~\cite{Do,Stevenson2007,CasconKreuzerNochettoSiebert2008,BeckerMaoShi2010,CarstensenRabus2011,HuangXu2012})
and more recently has been applied to
eigenvalue problems as well 
(see, e.g.,
\cite{GarauMorinZuppa2009,GianiGraham2009,CarstensenGedicke2011}
for convergence
and \cite{DaiXuZhou2008,CarstensenGedicke2012,DaiHeZhou2014,CarstensenGallistlSchedensack2015}
for optimal rates).
Some survey papers are
available; we refer, in particular, for further reading and references,
to~\cite{NochettoSiebertVeeser2009,NochettoVeeser,CarstensenFeischlPagePraetorius2014}.
In the case of eigenvalue approximation, it has been recently observed that
adaptive schemes driven by the error indicator associated to an individual
eigenvalue may produce unsatisfactory results, and that eigenvalues belonging
to clusters have to be considered simultaneously (see, in
particular,~\cite{Gallistl2014thesis,Gallistl2014nonconf,Gallistl2014conform}).

In this paper, we study the adaptive approximation of the Laplace
eigenvalue problem by mixed finite elements.
The analysis of the underlying formulation, which fits the framework of
$(0,g)$-type mixed problems, is not a mere
generalization of the case of standard conforming Galerkin approach
(see~\cite{Boffi2010}, where the convergence and
the a~priori estimates are recalled).
This causes additional technical difficulties which were in previous
works~\cite{DuranGastaldiPadra1999} circumvented by showing equivalence
with some nonconforming but elliptic finite element formulation.
Typically, residual-based a~posteriori error estimates are derived by
exploiting the fact that the error of the eigenvalues as well as
the error of the eigenfunctions in some weaker norm (usually the $L^2$
norm) is of higher-order compared with the error in the energy-like
norm.
The higher-order $L^2$ convergence, however, is not valid in its
original format in mixed FEMs, and one technical tool we make use
of is a fairly abstract superconvergence result for eigenvalue problems
where a certain error quantity is shown to be of higher order in
the $L^2$ norm.
For the low-order case a similar result was shown
in~\cite{Gardini2009} by using the representation in terms of
nonconforming finite elements from~\cite{DuranGastaldiPadra1999}.

We follow the argument
of~\cite{CasconKreuzerNochettoSiebert2008} in order to show the optimality of
an adaptive finite element scheme which is constructed taking into account
clusters of eigenvalues in the spirit of~\cite{Gallistl2014thesis}. In order
to obtain the result, we need to derive estimates which are essentially
different from the case of standard FEMs: this is one of the main
contributions of our paper.

Previous a~posteriori estimates for mixed formulation (source or eigenvalues
problem) mostly showed efficiency and reliability with respect to the vector
variable only (see~\cite{Alonso1996} and~\cite{ChenHolstXu2009,HuangXu2012};
other results in this context can be found
in~\cite{BraessVerfuerth1996,WohlmuthHoppe1999,GaticaMaischak2005,LovadinaStenberg2006,LarsonMalqvist2008}).
Estimates involving the scalar variable were present
in~\cite{DuranGastaldiPadra1999} (where, as already mentioned, the
equivalence with nonconforming schemes is exploited) and
in~\cite{Carstensen1997} (where the source problem is considered). Another
main contribution of our analysis is that we show optimality also with respect
to the scalar variable (see Definitions~\ref{def:distance} and~\ref{def:gap}).
This is performed by a suitable definition of the error
indicator (see Definitions~\ref{def:indicator} and~\ref{def:indicator3D});
this allows to prove the optimal convergence rate not only for the
eigenfunction but for the eigenvalues as well (see Section~\ref{se:eig}).

The outline of the paper is as follows: Section~\ref{se:setting} introduces
the problem we are dealing with, Section~\ref{se:algorithm} describes the
error indicators and our adaptive algorithm, Section~\ref{se:theorem} states
the main theorem of our paper, concerning the convergence of the adaptive
scheme in terms of a \emph{theoretical} error indicator which is equivalent to
the error indicator used for the design of the AFEM algorithm.
Section~\ref{se:eig} shows that the convergence of the error indicator, which
is related to the convergence of the eigenfunctions, actually implies the
convergence of the eigenvalues as well.
Finally, Section~\ref{se:lemmas} contains all technical results which are used
in the proof of our main theorem and Section~\ref{se:3D} discusses the
extension to three space dimensions.

Throughout this paper, we use standard notation for Lebesgue and
Sobolev spaces and their norms. The $L^2$ norm of a function $v$
over some domain $\omega$ is denoted by $\|v\|_\omega$
and, if there is no risk of confusion, we write $\|v\|=\|v\|_\Omega$
for the physical domain $\Omega$. The scalar product of $L^2(\Omega)$ is
denoted by $(\cdot,\cdot)$. If $\mathcal{A}$ is a disjoint union of
subdomains of $\Omega$, typically a (subset of a) triangulation, then
$\|v\|_\mathcal{A}^2=\sum_{\omega\in\mathcal{A}}\|v\|_\omega^2$.
We denote the scalar curl of some two-dimensional vector field $\psi$
by $\curl\psi=\partial_2\psi_1 -\partial_1\psi_2$
and the vector curl of a scalar-valued function $v$ by
$\ccurl v = (-\partial_2 v,\partial_1 v)^T$.
In three dimensions we define as usual $\ccurl\psi=\nabla\times\psi$.

The notation $A\lesssim B$ refers to an inequality
$A\leq C B$ up to a constant $C$ that is independent of the mesh size.
We do not trace the explicit dependence of the constants on the eigenvalues,
cf.\ Remark~\ref{r:constants}.

The mesh-size is typically denoted by $h$; when a triangulation $\T_h$ is
obtained as a refinement of a given mesh, we denote by $\T_H$ the coarser
mesh. When dealing with the adaptive scheme, we denote by $\ell$ the level of
refinement, so that $\T_{\ell+1}$ is the next triangulation in the algorithm
obtained from $\T_\ell$.

\section{Setting of the problem}
\label{se:setting}

Our main result is valid both in two and three dimensions. From now on, we
discuss the two dimensional setting. Section~\ref{se:3D} extends the result in
three dimensions.

Given a polygonal domain $\Omega$, in this paper we are interested in the
following eigenvalue problem associated with the Laplace operator in mixed
form: find $\lambda\in\RE$ and $u\in L^2(\Omega)$ with $\|u\|=1$
such that for some $\ssigma\in H(\ddiv;\Omega)$ it holds
\[
\left\{
\aligned
&\int_\Omega\ssigma\cdot\ttau\,d\x+\int_\Omega u\ddiv\ttau\,d\x=0&&
\forall\ttau\in H(\ddiv;\Omega)\\
&\int_\Omega v\ddiv\ssigma\,d\x=-\lambda\int_\Omega uv\,d\x&&
\forall v\in L^2(\Omega).
\endaligned
\right.
\]

\subsection{Abstract mixed eigenvalue problem}

We cast this problem within the standard setting of abstract eigenvalue
problems in mixed form of the second type (see~\cite{bbg2,Boffi2010}).

Let $\Sigma$, $M$, $\H$ be Hilbert spaces such that $M\subseteq\H\subseteq
M^\star$ and consider two bilinear and continuous forms
$a:\Sigma\times\Sigma\to\mathbb R$ symmetric, and
$b:\Sigma\times M\to\mathbb R$ which satisfy the usual hypotheses for mixed
problems~\cite{BoffiBrezziFortin2013}: $a$ is elliptic in the kernel of $b$
and $b$ fulfills the inf-sup condition. Moreover, the form $a$ is supposed to
be positive definite so that the associated norm $|\cdot|_a$ is well defined.
In the pivot space $\H$ we consider the scalar product $(\cdot,\cdot)_\H$ and
corresponding norm $\|\cdot\|_\H$.

In this framework, the continuous eigenvalue problem reads: find
$\lambda\in\RE$ and $u\in M$ with $\|u\|_\H=1$ such that for some
$\sigma\in\Sigma$ it holds
\begin{equation}
\left\{
\aligned
&a(\sigma,\tau)+b(\tau,u)=0&&\forall\tau\in\Sigma\\
&b(\sigma,v)=-\lambda(u,v)_\H&&\forall v\in M
\endaligned
\right.
\label{e:exactEVP}
\end{equation}
and, given finite dimensional subspaces $\Sigma_h\subset\Sigma$ and
$M_h\subset M$
(typically associated to a finite element mesh $\T_h$),
its discrete counterpart is: find
$\lambda_h\in\RE$ and $u_h\in M_h$ with $\|u_h\|_{\H}=1$ such that for some 
$\sigma_h\in\Sigma_h$ it holds
\begin{equation}
\left\{
\aligned
&a(\sigma_h,\tau)+b(\tau,u_h)=0&&\forall\tau\in\Sigma_h\\
&b(\sigma_h,v)=-\lambda_h(u_h,v)_\H&&\forall v\in M_h.
\endaligned
\right.
\label{e:discreteEVP}
\end{equation}

The following three assumptions ensure the good approximation of the
eigenmodes (see \cite{bbg2,Boffi2010}), where $\rho(h)$ tends to zero as $h$
goes to zero and $\Sigma_0$ and $M_0$ are the subspaces of $\Sigma$ and $M$,
respectively, containing all solutions to the source problem associated
with~\eqref{e:exactEVP} when the datum is in $\H$; the discrete kernel
associated to the bilinear form $b$ is as usual defined as
\[
\kernel=\{\tau\in\Sigma_h:b(\tau,v)=0\ \forall v\in M_h\}.
\]

\begin{description}

\item[Fortid condition.] There exists a Fortin operator
$\fortin:\Sigma_0\to\Sigma_h$ such that
\[
b(\sigma-\fortin\sigma,v)=0\quad\forall v\in M_h
\]
and
\[
|\sigma-\Pi_{F,h}\sigma|_a \leq \rho(h) \|\sigma\|_{\Sigma_0}
\quad\forall \sigma\in\Sigma_0.
\]

\item[Weak approximability of $M_0$.]
\[
b(\tau_h,v)\le\rho(h)|\tau_h|_a\|v\|_{M_0}\quad
\forall v\in M_0\ \forall\tau_h\in\kernel.
\]

\item[Strong approximability of $M_0$.]
\[
\inf_{v_h\in M_h} \|v-v_h\|_\H  \le\rho(h) \|v\|_{M_0}
\quad\forall v\in M_0.
\]

\end{description}

We consider a problem associated with a compact operator, so that the
eigenvalues are enumerated as
\begin{equation*}
 0<\lambda_1\leq\lambda_2\leq\lambda_3\leq\dots
\end{equation*}
(we repeat the eigenvalues according to their multiplicities); the
corresponding eigenfunctions are denoted by
$\{(\sigma_1,u_1),(\sigma_2,u_2),\dots\}$ and the $\{u_i\}$'s form an
orthonormal system in $\H$. In particular, we have $|\sigma_i|^2_a=\lambda_i$
and $\|u_i\|_\H=1$ for $i=1,2,\dots$.
We denote by $\eigl$ the span of the $\{u_i\}$'s corresponding to $\lambda$.

Analogously, the discrete eigenvalues can be enumerated as follows
\begin{equation*}
 0<\lambda_{h,1}\leq\lambda_{h,2}\leq\dots\le\lambda_{h,N(h)}
\end{equation*}
with corresponding eigenfunctions
$\{(\sigma_{h,1},u_{h,1}),\dots,(\sigma_{h,N(h)},u_{h,N(h)})\}$,
where $N(h)=\dim(M_h)$ and the $\{u_{h,i}\}$'s form an orthonormal system in
$\H$. Here we have $|\sigma_{h,i}|^2_a=\lambda_{h,i}$ and $\|u_{h,i}\|_\H=1$ for
$i=1,2,\dots,N(h)$.

For a cluster of eigenvalues $\lambda_{n+1},\dots,$ $\lambda_{n+\N}$ of length
$\N\in\mathbb{N}$, we define the index set $J=\{n+1,\dots,n+\N\}$ and the
spaces
\begin{equation*}
W = \operatorname{span}\{u_j\mid j\in J\} \quad\text{and}\quad
W_{\T_h}=W_h = \operatorname{span}\{u_{h,j}\mid j\in J\}.
\end{equation*}

\subsection{Some useful operators}

\begin{definition}
For any $w\in M$ we define $\G(w)\in\Sigma$ as the solution to
\begin{equation}
a(\G(w),\tau) + b(\tau, w) = 0
\quad\text{for all } \tau\in\Sigma.
\label{eq:G}
\end{equation}
For any $w_h\in M_h$ we define its discrete counterpart
$\G_h(w_h)\in\Sigma_h$ via
\begin{equation}
a(\G_h(w_h),\tau_h) + b(\tau_h, w_h) = 0
\quad\text{for all } \tau_h\in\Sigma_h.
\label{eq:Gh}
\end{equation}
We explicitly notice that when two meshes $\T_h$ and $\T_H$ are present, it
is important to distinguish between $\G_h$ and $\G_H$.

\end{definition}

In many applications and corresponding instances of $a$ and $b$,
the above definition is related to an integration by parts formula where 
$\G(w)$ is some derivative of $w$. For instance, in the case of mixed
Laplacian, $\G(w)$ is the gradient of $w$.

\begin{definition}

The solution operators $T:\H\to M$ and $A:\H\to\Sigma$ are defined by
\begin{equation}
\left\{
\aligned
&   a(Ag,\tau) + b(\tau,T g)  = 0 &&\forall \tau\in\Sigma\\
&   b(Ag,v) = - (g,v)_\H &&\forall v\in M
\endaligned
\right.
\label{eq:defTA}
\end{equation}
and $T_h:\H\to M_h$ and $A_h:\H\to\Sigma_h$ are their discrete counterparts
\begin{equation}
\left\{
\aligned
&   a(A_hg,\tau_h) + b(\tau_h,T_h g) = 0 &&\forall\tau_h\in\Sigma_h\\
&   b(A_hg,v_h) = - (g,v_h)_\H &&\forall v_h\in M_h.
\endaligned
\right.
\label{eq:defTAh}
\end{equation}

\end{definition}

\begin{definition}

The operator $\Tlh:\H\to M_h$ ($\lambda\in\RE$) is defined by
\begin{equation}\label{e:discreteSource}
\left\{
\aligned
&   a(\G_h(\Tlh g), \tau_h) + b(\tau_h, \Tlh g)  = 0 &&\forall\tau_h\in\Sigma_h\\
&   b(\G_h(\Tlh g), v_h ) =  - (\lambda g, v_h)_\H &&\forall v_h\in M_h,
\endaligned
\right.
\end{equation}
that is, $\Tlh=\lambda T_h$.

\end{definition}

Let $P^W_h$ denote the $\H$-orthogonal projection onto $W_h$. The
following definition is crucial for the definition of our theoretical error
indicator.

\begin{definition}
\label{def:Lambda}
The operator $\Lambda_h:\eigl\to W_h$ is defined as follows:
\[
\Lambda_h = P^W_h \circ \Tlh .
\]
\end{definition}

For the sake of simplicity, we do not include the dependence from $\lambda$
in the notation for $\Lambda_h$: it will be clear from the context that when
$\Lambda_h$ is applied to an element of $\eigl$, the corresponding value of
$\lambda$ should be used for its definition.

\begin{lemma}\label{l:commuting}
The operators $P^W_h$ and $\Tlh$ commute, that is
$\Lambda_h = P^W_h \circ \Tlh =  \Tlh \circ P^W_h$.
In other words, if $(\sigma,u)$ is an eigenfunction associated with $\lambda$,
then $\Lambda_h u$ solves

\begin{equation*}
\left\{
\aligned
&   a(\G_h(\Lambda_h u), \tau_h) + b(\tau_h, \Lambda_h u) = 0
&&\forall\tau_h\in\Sigma_h\\
&   b(\G_h(\Lambda_h u), v_h ) = - (\lambda P^W_h u, v_h)_\H &&\forall v_h\in M_h.
\endaligned
\right.
\end{equation*}
\end{lemma}
\begin{proof}
We adapt the result of~\cite[Lemma~2.2]{Gallistl2014conform}.
The expansion of $\Lambda_h u$ reads as
$ \Lambda_h u = \sum_{j\in J} (\Tlh u, u_{h,j})_\H u_{h,j}$,
thus $\Lambda_h u$ solves the discrete linear system~\eqref{eq:defTAh} with
right-hand side
$g = \sum_{j\in J} (\Tlh u, u_{h,j})_\H \lambda_{h,j}u_{h,j}$.
For any $j\in J$ we have
\begin{equation*}
\aligned
\lambda_{h,j} (\Tlh u, u_{h,j})_\H 
&= -b(\sigma_{h,j},\Tlh u)
=  a(\G_h(\Tlh u),\sigma_{h,j})
= - b(\G_h(\Tlh u),u_{h,j})\\
&= \lambda (u, u_{h,j})_\H,
\endaligned
\end{equation*}
which gives the final result that $\Lambda_h u$ solves the discrete linear
system~\eqref{eq:defTAh} with
right-hand side
$g = \sum_{j\in J}\lambda(u, u_{h,j})_\H u_{h,j}=\lambda P^W_hu$.

\end{proof}

\section{AFEM algorithm and error quantities}
\label{se:algorithm}

As already mentioned, we are interested in the Laplace eigenvalue problem in
mixed form with Dirichlet boundary conditions. Namely, with the notation
introduced in Section~\ref{se:setting}, we are making the following choices:
\[
\aligned
&\Sigma=H(\ddiv;\Omega)\\
&M=\H=L^2(\Omega)\\
&a(\sigma,\tau)=(\sigma,\tau)\\
&b(\tau,v)=(\ddiv\tau,v)
\endaligned
\]
for an open, bounded, simply-connected polygonal Lipschitz domain $\Omega$.

It follows, in particular that the seminorm $|\cdot|_a$ is the norm in
$(L^2(\Omega))^2$.
Our analysis applies to more general operators (for instance, Neumann boundary
conditions or non-constant coefficients), but we stick to this simpler
example for the sake of readability.

We discretize the problem with standard mixed finite elements (including
Raviart--Thomas, Brezzi--Douglas--Marini, etc.),
see~\cite{BoffiBrezziFortin2013} for more detail. It is well-known that this
choice satisfies the assumptions discussed in Section~\ref{se:setting} (see,
for instance,~\cite{bbg2}).

Moreover, we observe that the following relation (part of the commuting
diagram) holds true:
\begin{equation}
\ddiv(\Sigma_h) = M_h
\label{eq:sparita}
\end{equation}

Let us first introduce our error indicator.

\begin{definition}
Let $\T_h$ be a triangulation of $\Omega$ and let
$(\sigma_{h,j},u_{h,j})\in\Sigma_h\times M_h$ be a discrete eigensolution
computed on the mesh $\T_h$.
Then, for all $T\in\T_h$ we define
\begin{equation*}
\eta_{h,j}(T)^2
=
   \| h_T (\sigma_{h,j} - \nabla u_{h,j}) \|_T^2 
   + \| h_T \curl \sigma_{h,j} \|_T^2 
   + \sum_{E\in\mathcal E(T)} h_E \| [\sigma_{h,j}]_E \cdot t_E\|_E^2,
\end{equation*}
where $h_T$ is the diameter of $T$, $\mathcal{E}(T)$ denotes the set of edges
of $T$, $h_E$ is the length of the edge $E$, and $t_E$ is its unit tangent
vector. As usual, $[\sigma_h]_E\cdot t_E$ denotes the jump of the trace of
$\sigma_h\cdot t_E$ for internal edges and the trace for boundary edges.

Given a set $\mathcal{M}$ of elements of $\T_h$, we define
\[
\eta_{h,j}(\mathcal{M})^2=\sum_{T\in\mathcal{M}}\eta_{h,j}(T)^2.
\]
\label{def:indicator}
\end{definition}

\subsection{Adaptive algorithm}

The adaptive algorithm consists of the standard four steps: solve, estimate,
mark, and refine. In the description of the fours steps, we describe how the
algorithms runs from level $\ell$ to $\ell+1$.

\begin{description}

\item[Solve.] Given a mesh $\T_\ell$ the algorithm computes the eigensolutions
of~\eqref{e:discreteEVP} belonging to the cluster
$(\lambda_{\ell,j},\sigma_{\ell,j},u_{\ell,j})$ for $j\in J$. We assume that
the discrete solution is computed exactly.

\item[Estimate.] The algorithm computes the local contributions of the error
estimator for the eigenfunctions in the cluster
$\big\{\eta_{\ell,j}(T)\big\}_{T\in\T_\ell}$ ($j\in J$).

\item[Mark.] The algorithm uses the well known D\"orfler marking 
strategy \cite{Do}.
Given a bulk parameter $\theta\in(0,1]$, a minimal subset
$\mathcal{M}_\ell\subseteq\T_\ell$ is identified such that
\begin{equation*}
    \theta \sum_{j\in J}\eta_{\ell,j}(\T_\ell)^2
\le  \sum_{j\in J}\eta_{\ell,j}(\mathcal{M}_\ell)^2.
\end{equation*}
The elements belonging to $\mathcal{M}_\ell$ are marked for refinement.

\item[Refine.] A new triangulation $\T_{\ell+1}$ is generated, as the smallest
admissible refinement of $\T_\ell$ satisfying
$\mathcal{M}_\ell\cap\T_{\ell+1}=\emptyset$ by using the refinement rules
of~\cite{BDD,stevenson2008}. 
Figure~\ref{fig:NVB} shows possible refinements of a
triangle.

\end{description}

To summarize, the adaptive algorithm accepts as \textbf{input} the bulk
parameter $\theta$ and the initial mesh $\T_0$
(with proper initialization of refinement edges 
as in \cite{BDD,stevenson2008}),
and returns as \textbf{output}
a sequence of meshes $\{\T_{\ell}\}$ and of discrete eigenpairs
$\{(\lambda_{\ell,j},\sigma_{\ell,j},u_{\ell,j})\}_{j\in J}$.

\begin{figure}[tb]
\centering
\begin{tabular}{ccccc} 
\setlength{\unitlength}{2.5mm}
\begin{picture}(7,3.5)
\thicklines
\put(0,0){\line(1,0){7}}
\thinlines
\put(7,0){\line(-1,1){3.5}}
\put(0,0){\line(1,1){3.5}}
\end{picture}&
\setlength{\unitlength}{2.5mm}
\begin{picture}(7,3.5)
\thicklines
\put(7,0){\line(-1,1){3.5}}
\put(0,0){\line(1,1){3.5}}
\thinlines
\put(3.5,3.5){\line(0,-1){3.5}}
\put(0,0){\line(1,0){7}}
\end{picture}&
\setlength{\unitlength}{2.5mm}
\begin{picture}(7,3.5)
\thicklines
\put(0,0){\line(1,0){3.5}}
\put(3.5,0){\line(0,1){3.5}}
\put(3.5,3.5){\line(1,-1){3.5}}
\thinlines
\put(3.5,0){\line(1,0){3.5}}
\put(3.5,3.5){\line(-1,-1){3.5}}
\put(3.5,0){\line(-1,1){1.75}}
\end{picture} &

\setlength{\unitlength}{2.5mm}
\begin{picture}(7,3.5)
\thicklines
\put(3.5,0){\line(1,0){3.5}}
\put(3.5,0){\line(0,1){3.5}}
\put(0,0){\line(1,1){3.5}}
\thinlines
\put(0,0){\line(1,0){3.5}}
\put(7,0){\line(-1,1){3.5}}
\put(3.5,0){\line(1,1){1.75}}
\end{picture}&
\setlength{\unitlength}{2.5mm}
\begin{picture}(7,3.5)
\thicklines
\put(0,0){\line(1,0){7}}
\put(3.5,0){\line(0,1){3.5}}
\thinlines
\put(0,0){\line(1,1){3.5}}
\put(7,0){\line(-1,1){3.5}}
\put(1.75,1.75){\line(1,-1){1.75}}
\put(5.25,1.75){\line(-1,-1){1.75}}
\end{picture}
\end{tabular}
\caption{Possible refinements of a triangle $T$ in one level
in 2D. The thick lines indicate the refinement edges of 
the sub-triangles as in \cite{BDD,stevenson2008}.}
\label{fig:NVB}
\end{figure}
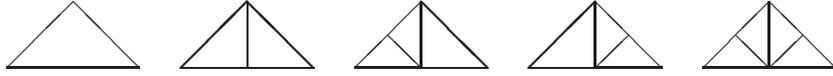

Finally, we shall make use of the following notation: given an initial mesh
$\T_0$, regular in the sense of Ciarlet, we denote by $\mathbb{T}$ the set of
\emph{admissible meshes} in the sense that a mesh in $\mathbb{T}$ is a
refinement of $\T_0$ obtained using the rules 
of~\cite{BDD,stevenson2008}. 

\subsection{Error quantities and theoretical error indicator}

The following definition introduces a metric 
in $M$.
\begin{definition}\label{def:distance}
$d:M\times M\to\RE$ is defined as
\[
d(v,w)=\sqrt{\|v-w\|^2+|\G(v)-\G(w)|_a^2}
\]
When $v$ (resp.\ $w$) belongs to $M_h$, then $\G_h(v)$ (resp.\ $\G_h(w)$)
should be used.

\end{definition}

\begin{remark}\label{r:constants}
We note that it may be useful to balance the terms in the square
root of Definition~\ref{def:distance} in terms of $\lambda$.
In particular, if $v$ and $w$ are related to eigenfunctions with frequency
$\lambda$, the right scaling would involve multiplying by $\lambda$ the term
$\|v-w\|$.
This is of particular interest it one aims to quantify the conditions
on the initial mesh-size.
In this paper, we do not aim at such a quantification
and refer the interested reader to \cite{Gallistl2014conform} for such a 
$\lambda$ explicit analysis in the context of conforming standard
finite elements.
\end{remark}

This distance allows us to evaluate the gap between discrete and continuous
eigenfunctions in the cluster.

\begin{definition}
\label{def:gap}
The following quantity measures how combinations of eigenfunctions in the
cluster $W$ are approximated by their discrete counterparts in $W_h$.
\[
\delta(W,W_h)
 = \sup_{\substack{u\in W \\ \| u\| = 1}} \inf_{v_h\in W_h} d(u,v_h).
\]

\end{definition}

Given a refinement $\T_\ell\in\mathbb{T}$ of the initial mesh $\T_0$, our
theory is based on the introduction of the following \emph{non-computable}
error indicator $\mu_{\ell}$ which will be proved equivalent to the
\emph{computable} indicator $\eta_\ell$.

\begin{definition}
Let $\T_h\in\mathbb{T}$ be a triangulation and for all $T\in\T_h$ and
$g_h\in M_h$ let us consider the following seminorm
\[
\begin{split}
|g_h|_{\eta,T}^2={}&
\| h_T (\G_h(g_h) - \nabla g_h) \|_T^2 
   + \| h_T \curl \G_h(g_h) \|_T^2 \\
&   + \sum_{E\in\mathcal E(T)} h_E \| [\G_h(g_h)]_E \cdot t_E\|_E^2,
\end{split}
\]
so that
\[
\eta_{h,j}(T)=|u_{h,j}|_{\eta,T}.
\]
Then, given an eigenfunction
$(\sigma,u)$ associated to the eigenvalue $\lambda$ (in particular, this is
used in the definition of $\Lambda_h$), we define
\begin{equation*}
\mu_h(u;T)=|\Lambda_h u|_{\eta,T}.
\end{equation*}
Given a set $\mathcal{M}$ of elements of $\T_h$, we define
\[
\mu_h(u;\mathcal{M})^2=\sum_{T\in\mathcal{M}}\mu_h(u;T)^2.
\]

\end{definition}

The next lemma is of technical nature and gives a criterion for linear
independence.
It generalizes \cite[Prop.~3.2]{CarstensenGedicke2011}.
\begin{lemma}\label{l:epsilonlemma}
Recall the notation $\N=\operatorname{card}(J)$ and suppose that
 \begin{equation}\label{e:epsilon}
 \varepsilon 
 = \max_{j\in J} \| u_j - \Lambda_h u_j \|
 \leq \sqrt{1 + 1/(2\N)} -1 .
 \end{equation}
 Then, 
 $\left\{ \Lambda_h u_j\right\}_{j\in J}$
 forms a basis of $W_h$.
 For any $w_h \in W_h$ with $\|w_h\|=1$,
 the coefficients of the representation
 $w_h = \sum_{j\in J}\gamma_j \Lambda_h u_j$
 are controlled as
\begin{equation}\label{e:ABbound}
  \sum_{j\in J} |\gamma_j|^2
  \leq  2+ 4 \N.
\end{equation}
 \end{lemma}%

\begin{proof}
The proof employs Gershgorin's theorem.
Since the proof follows verbatim the lines of
\cite[Lemma~5.1]{Gallistl2014conform}, it is omitted here.
\end{proof}

The following lemma states the equivalence between the two introduced
estimators. It is clear that the adaptive algorithm will make use of the
computable indicator $\eta$, while the indicator $\mu$ will be used for the
analysis.

\begin{lemma}[Local comparison of the error estimators]\label{l:estcomparison}
Provided the initial mesh-size is small enough such that 
\eqref{e:epsilon} is satisfied,
it holds for any $T\in\T_h$ that
\begin{equation*}
\N^{-1}
\sum_{j\in J} \mu_h(u_j;T)^2
\leq
\left(\frac{\B}{\A}\right)^2
\sum_{j\in J} \eta_{h,j}(T)^2
\leq
\left(\frac{\B}{\A}\right)^2
(2\N+4\N^2)
\sum_{j\in J} \mu_h(u_j;T)^2
\end{equation*}
where $[\A,\B]$ denotes a real interval containing the (continuous and
discrete) eigenvalue cluster and $\N$ is the number of eigenvalues in the
cluster.
\end{lemma}
\begin{proof}
The proof follows from a perturbation analysis as in
\cite[Prop.~5.1]{Gallistl2014conform}.
We include the proof for self-contained reading.
Let $k\in J$ and consider the 
expansion of
$\Lambda_h u_k = \sum_{j\in J} \gamma_j u_{h,j}$
with coefficients 
$\gamma_j  = (\Lambda_h u_k, u_{h,j})$.
The definition of $\Lambda_h$ and the symmetry yield
\begin{equation*}
\begin{aligned}
\gamma_j 
  &
  = (\Lambda_h u_k, u_{h,j})
  = (\Thl u_k, u_{h,j})
  =-\lambda_{h,j}^{-1}b(\sigma_{h,j}, \Thl u_k)
\\
  &
  = \lambda_{h,j}^{-1}a(\sigma_{h,j},\G_h(\Thl u_k))
  =-\lambda_{h,j}^{-1}b( \G_h(\Thl u_k),u_{h,j})
  = \lambda_{h,j}^{-1}\lambda_k (u_k,u_{h,j}) .
\end{aligned}
\end{equation*}
Since $\{u_{h,j}\}_{j\in J}$ is an orthonormal system, we arrive at
$\sum_{j\in J} \gamma_j^2 \leq (\B/\A)^2$, which implies
\begin{equation*}
  |\Lambda_h u_k|^2_{\eta,T}
 \leq 
  \bigg(\sum_{j\in J}\gamma_j^2\bigg)\sum_{j\in J}|u_{h,j}|^2_{\eta,T}\le
   \left(\frac{\B}{\A}\right)^2\sum_{j\in J}|u_{h,j}|^2_{\eta,T}.
\end{equation*}
This proves the first stated inequality.

Lemma~\ref{l:epsilonlemma} shows that there exist real coefficients 
$\{\delta_j\mid j\in J\}$   such that
\begin{equation*}
u_{h,k} = \sum_{j\in J} \delta_j \Lambda_h u_j
\quad\text{and}\quad
\sum_{j\in J} \delta_j^2 \leq 2+4\N  .
\end{equation*}
The triangle and Cauchy inequalities  lead to
\begin{equation*}
  |u_{h,k}|^2_{\eta,T}
   \leq
     \bigg(\sum_{j\in J} \delta_j^2 \bigg)
     \sum_{j\in J} |\Lambda_h u_j|^2_{\eta,T}
   \leq
   (2+4\N)    \sum_{j\in J}  |\Lambda_h u_j|^2_{\eta,T}.
 \end{equation*}
This shows the second stated inequality and concludes the proof.
\end{proof}

\section{Optimal convergence of the adaptive scheme}
\label{se:theorem}

In this section we state the main theorem showing the optimal convergence of
our adaptive scheme and sketch the principal lines of its proof. The structure
of the proof is closely related to~\cite{CasconKreuzerNochettoSiebert2008} and
relies on several intermediate results which, for the sake of 
readability, will be postponed to Section~\ref{se:lemmas}.

As usual in this context, the convergence is
measured by introducing a suitable nonlinear approximation class in the spirit
of~\cite{BDD}.
For any $m\in\mathbb{N}$, we denote by
\[
\mathbb T(m)=
  \{\T\in\mathbb T 
        \mid 
    \card(\T) - \card(\T_0)
               \leq m \}
\]
the set of admissible
triangulations in $\mathbb T$ whose cardinality differs from that
of $\T_0$ by $m$ or less.

The best algebraic convergence rate $s\in(0,+\infty)$ obtained by any
admissible mesh in $\mathbb{T}$ is characterized in terms of the following
seminorm
\begin{equation*}
 |W|_{\mathcal{A}_s} 
 = \sup_{m\in\mathbb{N}} m^s
       \inf_{\T\in\mathbb{T}(m)}  \delta(W,W_{\T}).
\end{equation*}
In particular, we have $|W|_{\mathcal{A}_s}<\infty$ if the rate of
convergence $\delta(W,W_{\T})=O(m^{-s})$ holds true for the optimal
triangulations $\T$ in $\mathbb{T}(m)$.

The main results of this section, stated in Theorem~\ref{th:main}, shows that
the same optimal rate of convergence is reached by the error quantity
$\delta(W,W_{\T_{\ell}})$ associated with the mesh sequence $\{\T_{\ell}\}$
obtained from the adaptive algorithm presented in Section~\ref{se:algorithm}.

\begin{theorem}
Provided the initial mesh-size and the bulk parameter $\theta$ are small
enough, if for the eigenvalue
cluster $W$ it holds $|W|_{\mathcal{A}_s}<\infty$, then the sequence of
discrete clusters $W_\ell$ computed on the mesh $\T_\ell$ satisfies the
optimal estimate
\begin{equation*}
\delta(W,W_\ell)
    (\card(\T_\ell) - \card(\T_0))^s
\lesssim |W|_{\mathcal{A}_s}.
\end{equation*}
\label{th:main}
\end{theorem}

\begin{proof}
We follow the lines of the proof of Theorem~3.1 in~\cite{Gallistl2014conform}.
The main arguments are the same as in~\cite{CasconKreuzerNochettoSiebert2008}.

Given a positive $\beta$, we consider the quantity
\[
  \xi_\ell^2=  \sum_{j\in J}\mu_\ell(u_j,\T_\ell)^2
                    + \beta \sum_{j\in J} d(u_j,\Lambda_\ell u_j)^2
\]
which will be used in the contraction argument of
Proposition~\ref{p:contraction}. We do not consider the trivial case
$\xi_0 = 0$.
Choose $0<\tau\leq |W|_{\mathcal{A}_s}^2/\xi_0^2$, 
and set $\varepsilon(\ell)=\sqrt{\tau}\,\xi_\ell$. 
Let $N(\ell)\in\mathbb{N}$ be minimal with the property
\begin{equation*}
  |W|_{\mathcal{A}_s}^2
 \leq \varepsilon(\ell)^2\,N(\ell)^{2s}.
\end{equation*}
It can be easily seen that $N(\ell)>1$, otherwise
\[
|W|_{\mathcal{A}_s}\le\varepsilon(\ell)
\]
but this, together with the definition of $\varepsilon(\ell)$, would violate
the contraction property of Proposition~\ref{p:contraction}.

From the minimality of $N(\ell)$ it turns out that
\begin{equation}\label{e:optNellBound}
N(\ell)
\leq 2 |W|_{\mathcal{A}_s}^{1/s}
            \varepsilon(\ell)^{-1/s}
\quad\text{for all }\ell\in\mathbb{N}_0.
\end{equation}
Let $\widetilde\T_\ell\in\mathbb{T}$ denote the
optimal triangulation of cardinality
$$
 \card(\widetilde\T_\ell)
 \leq\card(\T_0) + N(\ell)
$$
in the sense that the operator
$\widetilde\Lambda = \Lambda_{\widetilde\T_\ell}$ of
Definition~\ref{def:Lambda} with respect to the mesh $\widetilde\T_\ell$
satisfies
\begin{equation} \label{e:tildeToptimal}
  \sum_{j\in J} d(u_j, \widetilde\Lambda u_j)^2 
\leq 
    N(\ell)^{-2s} |W|_{\mathcal{A}_s}^2
\leq \varepsilon(\ell)^2.
\end{equation}
Let us consider the overlay $\widehat\T_\ell$, that is the smallest common
refinement of $\T_\ell$ and $\widetilde\T_\ell$, which is 
known \cite{CasconKreuzerNochettoSiebert2008} to satisfy
\begin{equation}\label{e:overlayProperty}
\card(\T_\ell \setminus\widehat\T_\ell)\le
  \card(\widehat\T_\ell)  
 - \card(\T_\ell)
\leq 
  \card(\widetilde\T_\ell) -
        \card(\T_0)
\leq 
   N(\ell) .
\end{equation}
This relation and \eqref{e:optNellBound}--\eqref{e:overlayProperty}
lead to
\begin{equation}\label{e:OptEst1}
  \card(\T_\ell \setminus\widehat\T_\ell)
  \leq N(\ell)
 \leq 2 |W|_{\mathcal{A}_s}^{1/s}
            \varepsilon(\ell)^{-1/s}.
\end{equation}
Let $\widehat\Lambda$ denote the operator $\Lambda_{\widehat\T_\ell}$
with respect to the mesh $\widehat\T_\ell$.

The following estimate
\begin{equation}\label{e:step2}
 \sum_{j\in J}d(u_j,\widehat\Lambda u_j)^2
 \leq  3\varepsilon(\ell)^2
\end{equation}
follows from the quasi-orthogonality (see Proposition~\ref{l:qoI})
applied to $\T_h=\widehat\T_\ell$ and $\T_H=\widetilde\T_\ell$. Indeed
\begin{equation*}
(1-\Cqo\rho(h_0))
\sum_{j\in J}d(u_j,\widehat\Lambda u_j)^2
\leq
(1+\Cqo\rho(h_0))
\sum_{j\in J}d(u_j,\widetilde\Lambda u_j)^2 .
\end{equation*}
Estimate~\eqref{e:step2} follows from the mesh-size condition
$\Cqo\rho(h_0)\leq 1/2$ and \eqref{e:tildeToptimal}.

We now show the existence of a constant $C_1$ such that
\begin{equation}
\sum_{j\in J}
 \mu_\ell(u_j,\T_\ell)^2
   \leq
    C_1
    \sum_{j\in J}
    \mu_\ell(u_j,\T_\ell\setminus\widehat\T_\ell)^2.
\label{eq:keyargument}
\end{equation}

From the triangle inequality and the discrete reliability 
(see Proposition~\ref{p:drel}) we obtain for any $j\in J$
\begin{equation*}
\begin{split}
d(u_j,\Lambda_\ell u_j)^2
&\leq
2d(u_j,\widehat \Lambda_\ell u_j)^2
 + 2d(\widehat \Lambda_\ell u_j, \Lambda_\ell u_j)^2
\\
&\leq
2d(u_j,\widehat \Lambda_\ell u_j)^2
+
2 \Cdrel^2 \mu_\ell(\T_\ell\setminus\widehat\T_\ell)^2\\
&\quad+
C\rho(h_0)^2 (d(u_j,\Lambda_\ell u_j)+ d(u_j,\widehat\Lambda_\ell
u_j))^2.
\end{split}
\end{equation*}
Provided the initial mesh-size is sufficiently small,
this leads to some constant $C_2$ such that with
\eqref{e:step2} it follows
\begin{equation*}
\sum_{j\in J}
d(u_j,\Lambda_\ell u_j)^2
\leq
C_2 \varepsilon(\ell)^2
+
C_2 \Cdrel^2 \sum_{j\in J}\mu_\ell(u_j,\T_\ell\setminus\widehat\T_\ell)^2.
\end{equation*}

Let $C_{\mathrm{eq}}$ denote the constant of
$C_2\xi_\ell^2 \leq C_{\mathrm{eq}} \sum_{j\in J} \mu_\ell(u_j,\T_\ell)^2$
(which exists by reliability).
The efficiency \eqref{e:efficiency}, the definition
of $\varepsilon(\ell)$, and the preceding estimates prove
\begin{equation*}
 \begin{aligned}
 C_{\mathrm{eff}}^{-2}
   \sum_{j\in J} \mu_\ell(u_j,\T_\ell)^2
 &
 \leq
 C_2 \varepsilon(\ell)^2
+
C_2 \Cdrel^2 \sum_{j\in J}\mu_\ell(u_j,\T_\ell\setminus\widehat\T_\ell)^2
 \\
 &
 \leq
    \tau C_{\mathrm{eq}}
    \sum_{j\in J}\mu_\ell(u_j,\T_\ell)^2
      + C_2\Cdrel^2 
         \sum_{j\in J} \mu_\ell(u_j,\T_\ell\setminus\widehat\T_\ell)^2
.
 \end{aligned}
\end{equation*}
Defining
$C_1 = 
  (C_{\mathrm{eff}}^{-2}-\tau C_{\mathrm{eq}})^{-1}
  C_2\Cdrel^2 
$,
which is positive for a sufficiently small choice of $\tau$, we
obtain~\eqref{eq:keyargument}.

In order to conclude the proof, we now make the following choice for the
parameter $\theta$:
\[
0<\theta\leq 1\left/ \left(C_1 (\B/\A)^2 (2\N^2+4\N^3)\right)\right..
\]

The marking step in the adaptive algorithm selects
$\mathcal{M}_\ell \subseteq \T_\ell$ with minimal
cardinality such that 
\[
\theta\sum_{j\in J}\eta_{\ell,j}(\T_\ell)^2
       \leq\sum_{j\in J}\eta_{\ell,j}(\mathcal M_\ell)^2.
\]
Estimate~\eqref{eq:keyargument} and the definition of $\theta$
imply together with Lemma~\ref{l:estcomparison}
that also $\T_\ell\setminus\widehat\T_\ell$
satisfies the bulk criterion, that is
\[
\theta\sum_{j\in J}\eta_{\ell,j}(\T_\ell)^2\le
\sum_{j\in J}\eta_{\ell,j}(\T_\ell\setminus\widehat\T_\ell)^2.
\]
The minimality of $\mathcal{M}_\ell$ and
\eqref{e:OptEst1} show that
\begin{equation}\label{e:minimalMl}
  \card(\mathcal M_\ell)
  \leq
  \card(\T_\ell\setminus\widehat\T_\ell)
  \leq
  2 |W|_{\mathcal A_s}^{1/s}
         \tau^{-1/(2s)} \xi_\ell^{-1/s}.
\end{equation}
It is proved in~\cite{BDD,stevenson2008}
that there exists a constant $C_{\mathrm{BDV}}$ such that
\begin{equation*}
 \begin{aligned}
  \card(\T_\ell)
   - \card(\T_0)
   &\leq
   C_{\mathrm{BDV}}
    \sum_{k=0}^{\ell-1} \card(\mathcal M_k)
   \\
   &\leq
   2C_{\mathrm{BDV}}   
   |W|_{\mathcal A_s}^{1/s}  
         \tau^{-1/(2s)} 
    \sum_{k=0}^{\ell-1}\xi_k^{-1/s}.
 \end{aligned}
\end{equation*}
The contraction property from 
Proposition~\ref{p:contraction} implies
$\xi_\ell^2 \leq \rho_2^{\ell-k} \xi_k^2$ for 
$k=0,\dots,\ell$.
Since $\rho_2<1$, a geometric series argument leads to
\begin{equation*}
 \sum_{k=0}^{\ell-1} \xi_k ^{-1/s}
 \leq
 \xi_\ell ^{-1/s} 
   \sum_{k=0}^{\ell-1} \rho_2^{(\ell-k)/(2s)}
 \leq
 \xi_\ell ^{-1/s} 
       \rho_2^{1/(2s)} \left/ \left(1-\rho_2^{1/(2s)}\right)\right..
\end{equation*}
The combination of the above estimates results in
\begin{equation*}
 \begin{aligned}
   &
     \card(\T_\ell)
    - \card(\T_0) 
   \\
   & \qquad\qquad
   \leq
      2C_{\mathrm{BDV}}   
      |W|_{\mathcal A_s}^{1/s} 
         \tau^{-1/(2s)}  \xi_\ell ^{-1/s} 
       \rho_2^{1/(2s)} \left/ \left(1-\rho_2^{1/(2s)}\right)\right..
 \end{aligned}
\end{equation*}
The equivalence of $\xi_\ell^2$ with the error 
 $\sum_{j\in J} d(u_j,\Lambda_\ell u_j)^2$ (reliability and efficiency, see
Section~\ref{se:lemmas}) concludes the proof.

\end{proof}

\section{Convergence of eigenvalues}
\label{se:eig}

The previous analysis shows that the adaptive procedure leads to the
convergence of the quantity $\delta(W,W_\ell)$ which is related to the
eigenfunctions belonging to the cluster.
In this section we show that this estimate actually implies the optimal
convergence of the eigenvalues.

The next discussion has been inspired by~\cite{dnr2}. However, we do not make
use explicitly of the spectral projections and follow a somehow more natural
argument (at least for symmetric problems).

As usual, we consider the eigenvalues $\mu_i=1/\lambda_i$ ($i=1,\dots$)
of $T$ and $\mu_{\ell,i}=1/\lambda_{\ell,i}$ ($i=1,\dots,\dim(M_\ell)$) of
$T_\ell$ and
discuss the convergence of $\mu_{\ell,j}$ to $\mu_j$ for $j\in J$. This standard
notation conflicts with our theoretical error indicator; nevertheless, we
believe that this overlap is not a source of confusion, since it is limited to
this section where the error indicator is not mentioned.

Let $E:\HH\to\HH$ denote the $\HH$ projection onto $W$ and
$E_\ell:\HH\to\HH$ the $\HH$ projection onto $W_\ell$.
We denote by $F_\ell$ the restriction of $E_\ell$ to $W$
\[
F_\ell=E_\ell|_W.
\]
The following proposition shows that for $\ell$ large enough the operator
$F_\ell$
is a bijection from $W$ to $W_\ell$ (which have the same dimension $\N$).

\begin{proposition}
\label{pr:Fh}
For $\ell$ large enough the operator $F_\ell$ is injective. Moreover,
$F^{-1}_\ell$ is
uniformly bounded in $\mathcal{L}(W_\ell,W)$ and
\[
\sup_{\substack{x\in W_\ell\\\|x\|_{\H}=1}}\|F^{-1}_\ell x-x\|_{\H}\le C\delta(W,W_\ell).
\]
\end{proposition}

\begin{proof}

It is enough to show that for $\ell$ sufficiently large
$\|F_\ell y-y\|_{\H}\le(1/2)\|y\|_{\H}$ for all $y\in W$ (see also~\cite[Lemma~2]{dnr2}).
Indeed, from the definition of $F_\ell$ it is immediate to get
\[
\|F_\ell y-y\|_{\H}\le\|y-y_\ell\|_{\H}\quad\forall y_\ell\in W_\ell
\]
which implies
\[
\|F_\ell y-y\|_{\H}\le\delta(W,W_\ell)\|y\|_{\H}.
\]
We can then conclude our proof from Theorem~\ref{th:main} observing that
$\delta(W,W_\ell)$ tends to zero.

\end{proof}

Let us define the following operators from $W$ into itself:
\[
\hat T=T|_W,\quad \hat T_\ell=F_\ell^{-1} T_\ell F_\ell.
\]
It is clear that the eigenvalues of $\hat T$ ($\hat T_\ell$, resp.) are equal to
$\mu_j$ ($\mu_{\ell,j}$ resp.), $j\in J$.

\begin{lemma}

The following estimates hold true for all $x\in W$
\begin{equation}
\aligned
\|(T-T_\ell)x\|_{\H}&\le C\delta(W,W_\ell),\\
|(A-A_\ell)x|_a&\le C\delta(W,W_\ell),\\
\|(A-A_\ell)x\|_\Sigma&\le C\delta(W,W_\ell).
\endaligned
\label{eq:trio}
\end{equation}
\label{le:bbf}
\end{lemma}

\begin{proof}

Let us denote $u=Tx$, $u_\ell=T_\ell x$, $\sigma=\G(u)=Ax$, and
$\sigma_\ell=\G_\ell(u_\ell)=A_\ell x$.

In order to prove the first estimate, we use a standard duality argument
and introduce the following auxiliary problem: find $\zeta\in\Sigma$ and
$w\in M$ such that
\[
\left\{
\aligned
&   a(\zeta,\tau) + b(\tau,w)   = 0 &&\forall\tau\in\Sigma\\
&   b(\zeta,v)  = - (u-u_\ell,v)_{\H} &&\forall v\in M.
\endaligned
\right.
\]
We clearly have $\|\zeta\|_\Sigma+\|w\|_M\le C\|u-u_\ell\|_{\H}$. By standard
arguments we get
\begin{equation}
\aligned
\|u-u_\ell\|^2_{\H}&=(u-u_\ell,u-u_\ell)_{\H}=-b(\zeta,u-u_\ell)\\
&=-b(\zeta-\Pi_{F,\ell}\zeta,u)-b(\Pi_{F,\ell}\zeta,u-u_\ell)\\
&=a(\sigma,\zeta-\Pi_{F,\ell}\zeta)+a(\G(u)-\G_\ell(u_\ell),\Pi_{F,\ell}\zeta).
\endaligned
\label{eq:stimaT}
\end{equation}
For all $v_\ell\in M_\ell$, the first term can be estimated as follows:
\[
\aligned
|a(\sigma,\zeta-\Pi_{F,\ell}\zeta)|&=
|a(\sigma-\G_\ell(v_\ell),\zeta-\Pi_{F,\ell}\zeta)+
a(\G_\ell(v_\ell),\zeta-\Pi_{F,\ell}\zeta)|\\
&=|a(\sigma-\G_\ell(v_\ell),\zeta-\Pi_{F,\ell}\zeta)-
b(\zeta-\Pi_{F,\ell}\zeta,v_\ell)|\\
&=|a(\sigma-\G_\ell(v_\ell),\zeta-\Pi_{F,\ell}\zeta)|\\
&\le C|\sigma-\G_\ell(v_\ell)|_a\|\zeta\|_\Sigma
\le C|\sigma-\G_\ell(v_\ell)|_a\|u-u_\ell\|_{\H}.
\endaligned
\]

The second term in the last line of~\eqref{eq:stimaT} can be estimated as
follows:
\[
|a(\G(u)-\G_\ell(u_\ell),\Pi_{F,\ell}\zeta)|\le C
|\G(u)-\G_\ell(u_\ell)|_a\|\zeta\|_\Sigma\le
C|(A-A_\ell)x|_a\|u-u_{\ell}\|_{\H}.
\]

Hence
\[
\|Tx-T_\ell x\|_{\H}\le C\left(|\sigma-\G_\ell(v_\ell)|_a+|(A-A_\ell)x|_a\right)\quad
\forall v_\ell\in M_\ell.
\]
Since the first term is bounded by $\delta(W,W_\ell)$, the final estimate will
follow from the second estimate in~\eqref{eq:trio}.

Let us prove the second estimate in~\eqref{eq:trio}.

From the definition of $W$ we have
\[
x=\sum_{j\in J}\alpha_j u_j,
\]
where we recall that $(\lambda_j,\sigma_j,u_j)$ is the generic eigensolution
belonging to the cluster $W$ and the coefficients are given by
$\alpha_j=(x,u_j)$.

Hence, $Ax=\G(u)$ with $u=Tx$ and
\[
Ax=\sum_{j\in J}\frac{1}{\lambda_j}\alpha_j\sigma_j.
\]

Analogously, from~\eqref{e:discreteSource},
\[
A_\ell x=\sum_{j\in J}\frac{1}{\lambda_j}\alpha_j\G_\ell(\Tll u_j).
\]

We then obtain
\[
|Ax-A_\ell x|_a=\left\vert\sum_{j\in J}\frac{1}{\lambda_j}\alpha_j
(\sigma_j-\G_\ell(\Tll u_j))\right\vert_a.
\]
We now show that $|\sigma_j-\G_\ell(\Tll u_j)|_a$ can be bounded by
$\delta(W,W_\ell)$. For all $v_\ell\in M_\ell$ we have
\[
\aligned
|\sigma_j-\G_\ell(\Tll u_j)|_a^2&=
a(\sigma_j-\G_\ell(\Tll u_j),\sigma_j-\G_\ell(\Tll u_j))\\
&=a(\sigma_j-\G_\ell(\Tll u_j),\sigma_j-\G_\ell(v_\ell))\\
&\quad+a(\sigma_j-\G_\ell(\Tll u_j),\G_\ell(v_\ell)-\G_\ell(\Tll u_j)).
\endaligned
\]
Since the last term is vanishing for the properties of $\sigma_j$ and the
definitions of $\Tll$ and $\G_\ell$, we obtain
\[
|\sigma_j-\G_\ell(\Tll u_j)|_a\le
\inf_{v_\ell\in M_\ell}|\sigma_j-\G_\ell(v_\ell)|_a\le
C\delta(W,W_\ell).
\]

From~\cite[Prop.~4.3.4]{BoffiBrezziFortin2013} and the definitions of $A$ and
$A_\ell$ it follows that
\[
\|(A-A_\ell)x\|_\Sigma\le C|(A-A_\ell)x|_a+C\|x-x_\ell\|_{\H},
\]
where $x_\ell\in M_\ell$ is the $\HH$ projection of $x$. The first term is
readily
bounded by $\delta(W,W_\ell)$, while the second one is smaller than
$\|x-F_\ell x\|_{\H}$
which has been already estimated in the proof of Proposition~\ref{pr:Fh}.

\end{proof}

The following proposition is a crucial
step for the bound of the eigenvalues.

\begin{proposition}

The following estimate holds true
\begin{equation}
\|\hat T-\hat T_\ell\|_{\mathcal{L}(W)}\le C\delta(W,W_\ell)^2
\label{eq:hatT}
\end{equation}

\end{proposition}

\begin{proof}

Let us define $\mathcal{S}_\ell=F_\ell^{-1}E_\ell-I: \H\to\H$. From the
boundedness of the involved operators, it is immediate to observe that
$\mathcal{S}_\ell$
is
uniformly bounded.

For all $x\in W$ we have
\begin{equation}
(\hat T-\hat T_\ell)x=(T-T_\ell)x+\mathcal{S}_\ell(T-T_\ell)x
\label{eq:duepez}
\end{equation}
since $E_\ell\mathcal{S}_\ell=0$. Let us estimate the first term. For all $x,y\in
W$ with $\|x\|_{\H}=\|y\|_{\H}=1$
\[
\aligned
((T-T_\ell)x,y)_{\H}&=-b(Ay,(T-T_\ell)x)+a((A-A_\ell)x,A_\ell y)+b(A_\ell
y,(T-T_\ell)x)\\
&=-b((A-A_\ell)y,(T-T_\ell)x)+a((A-A_\ell)x,A_\ell y).
\endaligned
\]

The first term is bounded by a constant times $\delta(W,W_\ell)^2$, while the
second term can be estimated as follows.
\[
\aligned
a((A-A_\ell)x,A_\ell y)&=a((A-A_\ell)x,(A_\ell-A)y)+a((A-A_\ell)x,Ay)\\
&=a((A-A_\ell)x,(A_\ell-A)y)-b((A-A_\ell)x,Ty)\\
&=a((A-A_\ell)x,(A_\ell-A)y)-b((A-A_\ell)x,(T-T_\ell)y)\\
&\le C\delta(W,W_\ell)^2.
\endaligned
\]

The second term in~\eqref{eq:duepez} can be estimated using the following
identity
\[
(\mathcal{S}_\ell(T-T_\ell)x,y)_{\H}=(\mathcal{S}_\ell(T-T_\ell)x,y-F_\ell y)_{\H}
\]
which finally leads to
\[
|(\mathcal{S}_\ell(T-T_\ell)x,y-E_\ell
y)_{\H}|\le\|\mathcal{S}_\ell\|_{\mathcal{L}(\H)}\|T-T_\ell\|_{\mathcal{L}(\H)}\|I-F_\ell\|_{\mathcal{L}(\H)}.
\]

\end{proof}

The operators $\hat T$ and $\hat T_\ell$ can be represented by symmetric positive
definite matrices of dimension $\N\times \N$ ($\N$ being the dimension of $W$).
The following theorem is then a standard consequence of matrix perturbation
theory (see, for instance,~\cite[Theorem~3, items c) and d)]{dnr2}) and to the
equivalences $\lambda_i=1/\mu_i$ and $\lambda_{\ell,i}=1/\mu_{\ell,i}$.

\begin{theorem}

Let $J$ denote the set of indices corresponding to the eigenvalues in the
cluster $W$. Then
\[
\sup_{i\in J}\inf_{j\in J}|\lambda_i-\lambda_{\ell,j}|\le C\delta(W,W_\ell)^2.
\]
\end{theorem}

\section{Auxiliary results}
\label{se:lemmas}

This section contains all technical results which have been used for the proof
of Theorem~\ref{th:main}. We arrange the presentation in three subsections: in
the first one a superconvergence result is proved; in the
second one we collect the results which hold for all refinements $\T_h$ of
a given mesh $\T_H$; finally, in the third one we include the results which
have been proved for the sequence of meshes $\{\T_\ell\}$ generated by our
adaptive procedure.

\subsection{A superconvergence result and other useful estimates}

Let $\Pi_h$ denote the orthogonal projection onto $M_h$.

\begin{lemma}[Superconvergence for the source
problem]\label{l:superconvSource}
There exist $\rho(h)$ tending to zero as $h$ goes to zero such that
\begin{equation*}
\| \Pi_h u - \Tlh u\|
\lesssim
\rho(h) \|\sigma - \G_h (\Tlh u) \|_{\Sigma}.
\end{equation*}
\end{lemma}
\begin{proof}
This result has been proved in~\cite{DR} and can be found
in~\cite{Gardini2009} or~\cite[\S 7.4]{BoffiBrezziFortin2013} as well.
\end{proof}

Let $J^C=\{1,\dots,N(h)\}\setminus J$ denote the indices of the discrete eigenvalues not
belonging to the cluster and assume the initial mesh-size is small enough
such that
\begin{equation*}
 K = 
 \sup_{\T_h}\sup_{k \in J^C}\sup_{j\in J}
    \frac{\lambda_j}{ \lvert \lambda_j-\lambda_{h,k} \rvert}
   < \infty.
\end{equation*}
\begin{lemma}\label{l:identityA}
For all $j\in J^C$ we have
\begin{equation*}
(u_{h,j},\Tlh u) 
 =
\frac{\lambda}{\lambda-\lambda_{h,j}}( \Tlh u- \Pi_h u, u_{h,j}) .
\end{equation*}
\end{lemma}
\begin{proof}
 We have
 \begin{equation*}
\aligned
 -\lambda_{h,j} (u_{h,j}, \Tlh u)
& =
 b( \sigma_{h,j},  \Tlh u)
 =
 -a(\sigma_{h,j},\G_h(\Tlh u))
 =
 b(\G_h(\Tlh u), u_{h,j})
 \\
& =
 -\lambda (u, u_{h,j})=
 -\lambda (\Pi_h u, u_{h,j}). 
\endaligned
 \end{equation*}
Adding $\lambda(u_{h,j} ,  \Tlh u)$ on both sides of this identity
leads to
 \begin{equation*}
 (\lambda-\lambda_{h,j}) (u_{h,j}, \Tlh u)
 =
 \lambda (\Tlh u - \Pi_h u, u_{h,j}). 
 \end{equation*}
\end{proof}

\begin{lemma}[cf.\ \cite{StrangFix1973}]\label{l:Kestimate}
Any eigensolution $(\lambda,\sigma,u)\in\mathbb R\times\Sigma\times W$
in the cluster satisfies
\begin{equation*}
\| \Tlh u - \Lambda_h u \| \leq K \| \Pi_h u - \Tlh u\| .
\end{equation*}
\end{lemma}
\begin{proof}
Let us define $e_h = \Tlh u - \Lambda_h u $. The expansion in terms of the
orthonormal basis $\{u_{h,j} \mid j=1,\dots,N(h)\}$ reads
\begin{equation*}
 e_h = \sum_{j\in J^C} \alpha_j u_{h,j}
 \quad
 \text{with }
  \sum_{j\in J^C} \alpha_j^2 = \| e_h \|^2 .
\end{equation*}
This relation, Lemma~\ref{l:identityA}, and Parceval's identity lead to
\begin{equation*}
\aligned
\| e_h \|^2
&=
\sum_{j\in J^C}  \alpha_j (\Tlh u, u_{h,j})
=
\sum_{j\in J^C}  \alpha_j \frac{\lambda}{\lambda-\lambda_{h,j}}(\Tlh u - \Pi_h
u, u_{h,j})\\
&\leq 
K \bigg(\sum_{j\in J^C}  \alpha_j^2\bigg)^{1/2} \|\Tlh u -\Pi_h u \|.
\endaligned
\end{equation*}
\end{proof}

We are now ready to prove the superconvergence result for the eigenvalue
problem.

\begin{proposition}[Superconvergence for the eigenvalue problem]\label{p:hot}
Any eigensolution $(\lambda,\sigma,u)\in\mathbb R\times\Sigma\times W$
in the cluster satisfies
\begin{equation*}
 \|\Pi_h u - \Lambda_h u\|
 \lesssim 
  \rho(h) \| \sigma - \G_h(\Tlh u) \|_\Sigma .
\end{equation*}
\end{proposition}
\begin{proof}
The triangle inequality and Lemma~\ref{l:Kestimate} give
\begin{equation*}
\|\Pi_h u - \Lambda_h u\|
  \leq \| \Pi_h u - \Tlh u \| + \|\Tlh u - \Lambda_h u\|
\le (1+K)\|\Pi_hu-\Tlh u\|.
\end{equation*}
The right-hand side has been estimated in
Lemma~\ref{l:superconvSource}.
\end{proof}

The following result contains a useful bound of the norm of the error in
$\Sigma$ in terms of our error quantity.

\begin{lemma}[Bound for the $\Sigma$ norm]\label{l:sigmanorm}
Any eigensolution $(\lambda,\sigma,u)\in\mathbb R\times\Sigma\times M$
satisfies
\begin{equation}
\label{eq:SigmaNormBound1}
\| \sigma - \G_h(\Lambda_h u)\|_\Sigma
\lesssim 
| \sigma - \G_h(\Lambda_h u) |_a 
  + (1+\lambda) \|u - \Lambda_h u\| .
\end{equation}
\end{lemma}

\begin{proof}
The stability of the continuous problem implies
\begin{equation*}
\begin{aligned}
&
\| \sigma - \G_h(\Lambda_h u)\|_\Sigma
\\
&\qquad
\lesssim 
\sup_{\substack{(\tau,v)\in \Sigma\times M\\
                \|\tau\|_\Sigma + \|v\| =1}}
\big( a(\sigma - \G_h(\Lambda_h u),\tau)
      + b(\sigma - \G_h(\Lambda_h u), v) 
      + b(\tau, u-\Lambda_h u)      
\big).
\end{aligned}
\end{equation*}
The identity~\eqref{eq:sparita} together with the continuous and discrete
eigenvalue problems imply
\begin{equation*}
\aligned
b(\sigma - \G_h(\Lambda_h u), v)
&=
b(\sigma,v) - b(\G_h(\Lambda_h u), \Pi_h v)
=
\lambda\big( (P^W_h u, \Pi_h v) - (u,v)\big)\\
&= \lambda(P^W_h u - u, v) .
\endaligned
\end{equation*}
Estimate~\eqref{eq:SigmaNormBound1} then follows from the continuity of $a$
and $b$
together with the elementary estimate
$\|u-P^W_h u\| \leq \|u-\Lambda_h u\|$.

\end{proof}

\subsection{Properties valid for all refinements $\T_h$ of $\T_H$}

We start this section by proving the efficiency of our theoretical error
estimator on the generic mesh $\T_h$.

\begin{proposition}[Efficiency]
\label{l:efficiency}
Let $(\sigma,u)$ be an eigenpair associated to the eigenvalue $\lambda$, then
there exists a positive constant $C_{\mathrm{eff}}$, independent of $h$, such
that
\begin{equation}
\label{e:efficiency}
\mu_h(u;\T_h)\le C_{\mathrm{eff}}d(u,\Lambda_h u).
\end{equation}
\end{proposition}
\begin{proof}
For the reader's convenience, we recall the definition of the error indicator
$\mu_h(u;T)$ for a given element $T\in\T_h$:
\begin{equation}
\aligned
\mu_{h}(u;T)^2
={}&
   \| h_T (\G_h(\Lambda_h u) - \nabla \Lambda_h u) \|_T^2 
   + \| h_T \curl \G_h(\Lambda_h u) \|_T^2 \\
&   + \sum_{E\in\mathcal E(T)} h_E \| [\G_h(\Lambda_h u)]_E \cdot t_E\|_E^2
.
\endaligned
\end{equation}
Following the same arguments as
in \cite[Lemma~6.3]{Carstensen1997}, we can
prove that
\begin{equation}
\label{e:eff1}
h_T^2 \|G_h(\Lambda_h u) - \nabla \Lambda_h u\|_{T}^2 \lesssim 
d(u,\Lambda_h u)^2.
\end{equation}
Finally, arguing as in the proof of Theorem 3.1 in~\cite{Alonso1996}, we can
bound the remaining terms of the error indicator as follows:
\begin{equation}
\label{e:eff2}
\| h_T \curl \G_h(\Lambda_h u) \|_T^2 + \sum_{E\in\mathcal E(T)} h_E \|
[\G_h(\Lambda_h u)]_E \cdot t_E\|_E^2 
\lesssim \|\sigma - G_h(\Lambda_h u)\|_{\tilde{T}}^2,
\end{equation}
where $\tilde{T}$ denotes the union of $T$ and its neighboring elements.\\
We then obtain the desired result by summing equations~(\ref{e:eff1})
and~(\ref{e:eff2}) over each elements $T\in\T_h$.

\end{proof}

The next result shows a uniform discrete reliability of the theoretical error
estimator when evaluated on the mesh $\T_h$, refinement of $\T_H$.

First of all, we recall the well-known discrete Helmholtz decomposition which
is valid for the finite element spaces we are considering. Suitable references
for this result are~\cite{DougActa} in the framework of discrete exterior
calculus or~\cite{GR}.
In our setting the discrete Helmholtz decomposition reads
(see~\cite[Lemma~2.5]{HuangXu2012}): for any
$\zeta_h\in\Sigma_h$ there exist $\alpha_h\in M_h$ and
$\beta_h\in\mathcal P_{k+1}(\T_h)$ (the space of continuous piecewise
polynomial
of degree $k+1$) such that
\begin{equation}
 \zeta_h = \G_h(\alpha_h) + \ccurl \beta_h.
\label{eq:Helmholtz}
\end{equation}
In particular, $\alpha_h\in M_h$ is such that
\begin{equation}
\label{eq:Helm}
 \aligned
&a(\G_h(\alpha_h),\tau_h)+b(\tau_h,\alpha_h)=0&&\forall\tau_h\in\Sigma_h\\
&b(\G_h(\alpha_h),v_h)=b(\zeta_h,v_h)&&\forall v_h\in M_h.
\endaligned
\end{equation}
By definition of the bilinear form and the fact that $\ddiv\Sigma_h=M_h$, we
have that $\ddiv(\G_h(\alpha_h)-\zeta_h)=0$, hence
$\G_h(\alpha_h)-\zeta_h=\ccurl\beta_h$.
Using again~\eqref{eq:sparita} there exists
$\hat\tau_h\in\Sigma_h$ such that $\ddiv\hat\tau_h=\alpha_h$.
From the discrete inf-sup condition we have
$\|\hat\tau_h\|\le C\|\alpha_h\|$.
Hence
\[
\|\alpha_h\|^2=b(\hat\tau_h,\alpha_h)=a(\G_h(\alpha_h),\hat\tau_h)
\le |\G_h(\alpha_h)|_a\|\hat\tau_h\|\le C|\G_h(\alpha_h)|_a\|\alpha_h\|,
\]
from which we obtain
\begin{equation}
\label{eq:stimaHelm}
\|\alpha_h\|\le C|\G_h(\alpha_h)|_a.
\end{equation}

\begin{proposition}[Discrete reliability]
  \label{p:drel}
Provided the mesh-size of $\T_H$ is sufficiently small, we have
\begin{equation*}
\aligned
&|\G_h(\Lambda_h u)-\G_H(\Lambda_H u)|_a +
\| \Lambda_h u - \Lambda_H u\|\\
&\qquad\le\Cdrel
\mu_H (u;\T_H \setminus \T_h)+C
\rho(H) (d(u,\Lambda_h u) + d(u,\Lambda_H u) ).
\endaligned
\end{equation*}
\end{proposition}

\begin{proof}
From the discrete Helmholtz decomposition \eqref{eq:Helmholtz}
there exist
$\alpha_h \in M_h$ and $\beta_h\in\mathcal P_{k+1}(\T_h)$
such that
\begin{equation*}
\G_h(\Lambda_h u) - \G_H(\Lambda_H u) 
     = \G_h(\alpha_h) + \ccurl \beta_h.
\end{equation*}
The term $\|\ccurl\beta_h\|$ can be bounded by using standard arguments as
in~\cite{DuranGastaldiPadra1999,BM,HuangXu2012}.
Actually, taking $\overline{\beta}_H$ as the Scott-Zhang interpolant
\cite{ScottZhang1990} of
$\beta_h$ on the mesh $\T_H$,
\[
\aligned
|\ccurl\beta_h|^2_a&=a(\G_h(\Lambda_h u)-\G_H(\Lambda_H u),\ccurl\beta_h)\\
&=-a(\G_H(\Lambda_H u),\ccurl(\beta_h-\overline{\beta}_H))\\
&=\sum_{T\in\T_H\setminus\T_h}\Big(
\int_T(\beta_h-\overline{\beta}_H)\curl\G_H(\Lambda_H u)\,dx\\
&\quad-
\int_{\partial T}(\beta_h-\overline{\beta}_H)\G_H(\Lambda_H u)\cdot t\,ds
\Big).
\endaligned
\]
Standard estimates for the Scott-Zhang interpolant give
\[
|\ccurl\beta_h|_a\lesssim\mu_H (u;\T_H\setminus\T_h).
\]
The integration by parts and some straightforward algebraic
manipulations lead to
\begin{equation*}
\begin{aligned}
|\G_h(\alpha_h)|_a^2%
&=
a(\G_h(\Lambda_h u)-\G_H(\Lambda_H u), \G_h(\alpha_h))
\\
&
= \lambda (P^W_hu-P^W_Hu, \alpha_h)
\\
&
= \lambda \big( (P^W_hu-\Pi_h u, \alpha_h)
+(\Pi_h u - \Pi_H u, \alpha_h -\Pi_H \alpha_h)\\
&\quad+(\Pi_H u -P^W_Hu, \alpha_h) \big).
\end{aligned}
\end{equation*}
We observe that $\|P_h^Wu-\Pi_hu\|\le\|\Lambda_hu-\Pi_hu\|$; indeed, $P_h^Wu$
is the best $H$-approximation of $u$ into $W_h$ and is characterized by
$(P_h^Wu-u,v_h)=(P_h^Wu-\Pi_hu,v_h)=0$ for all $v_h\in W_h$.
Hence, the estimate~\eqref{eq:stimaHelm}, Proposition~\ref{p:hot},
and Lemma~\ref{l:sigmanorm} prove for the first and the last term that
\begin{equation*}
\begin{aligned}
&(P^W_hu-\Pi_h u, \alpha_h)
+(\Pi_H u -P^W_Hu, \alpha_h)\\
&\qquad\lesssim
\big(\|P^W_hu-\Pi_h u\| + 
\|\Pi_H u -P^W_Hu\|\big)
|\G_h(\alpha_h)|_a
\\
&\qquad
\lesssim
\rho(H)(d(u,\Lambda_h u)+d(u,\Lambda_H u))
|\G_h(\alpha_h)|_a .
\end{aligned}
\end{equation*}
For the analysis of the remaining term, set
$\xi=\alpha_h -\Pi_H \alpha_h$.
It is shown in \cite[Lemma~2.8 and Equation~(3.9)]{HuangXu2012} that
$\xi$ satisfies
$\|\xi\| \lesssim H |\G_h(\alpha_h)|_a$.
Thus, we have with Proposition~\ref{p:hot} that
\begin{equation*}
\begin{aligned}
(\Pi_h u - \Pi_H u, \xi)
&
=
(\Pi_h u - \Lambda_H u, \xi)
\\
&
=
(\Pi_h u - \Lambda_h u, \xi)
+
(\Lambda_h u - \Lambda_H u, \xi)
\\
&
\lesssim 
(\rho(H)d(u,\Lambda_h u) 
+ H \|\Lambda_h u - \Lambda_H u\|)
|\G_h(\alpha_h)|_a .
\end{aligned}
\end{equation*}
Altogether we obtain for the error in the vector variable that
\begin{equation*}
\begin{aligned}
&
|\G_h(\Lambda_h u)-\G_H(\Lambda_H u)|_a
\\
&\qquad
\lesssim
\mu_H (u;\T_H\setminus\T_h)
+
\rho(H) (d(u,\Lambda_h u) +d(u,\Lambda_H u))+ H \|\Lambda_h u - \Lambda_H u\|.
\end{aligned}
\end{equation*}

It remains to estimate the error in the scalar variable.

Let $\hat z$ be the gradient of the solution $\hat\phi$ of the problem
\[
\aligned
&\Delta\hat\phi=\Lambda_h u - \Lambda_H u&&\text{in }\Omega\\
&\hat\phi=0&&\text{on }\partial\Omega.
\endaligned
\]

Using a (non-orthogonal) stable decomposition like the ones adopted
in~\cite[Lemma~3.3]{GaticaOyarzuaSayas2011} or~\cite[Lemma~2.1]{pasciak}, it
is possible to find $z\in H^1(\Omega)$ such that
\[
\hat z=z+\ccurl\psi.
\]

In particular we have
\begin{equation*}
\aligned
&\ddiv z = \Lambda_h u - \Lambda_H u\\
&\|z\|+\|\nabla z\|
 \lesssim \|\Lambda_h u - \Lambda_H u\|.
\endaligned
\end{equation*}
It follows
\begin{equation}
\label{e:uguale}
\aligned
\|\Lambda_h u-\Lambda_H u\|^2&=b(z,\Lambda_h u-\Lambda_H u)\\
&=b(\fortin z,\Lambda_h u)-b(\fortinH z,\Lambda_H u)\\
&=-a(\G_h(\Lambda_h u),\fortin z)+a(\G_H(\Lambda_H u),\fortinH z)\\
&=a(\G_H(\Lambda_H u)-\G_h(\Lambda_h u),\fortin z)+
a(\G_H(\Lambda_H u),(\fortinH-\fortin)z)
\\
&\le|\G_H(\Lambda_H u)-\G_h(\Lambda_h u)|_a\|\fortin z\|\\
&\quad+
a(\G_H(\Lambda_H u)-\nabla_H(\Lambda_H u),(\fortinH-\fortin)z),
\endaligned
\end{equation}
where we have used the definition of the Fortin operators $\fortin$,
$\fortinH$, of $\G_h$ and $\G_H$, and, in the last term, the
fact that the quantity 
$a(\nabla_H(\Lambda_H u),(\fortinH-\fortin)z)$ 
vanishes.

We observe furthermore that $\fortin z-\fortinH z = 0$ on the unrefined
elements $\T_H\cap\T_h$.
Since $z$ is smooth enough to allow for stability and first-order
approximation of $\fortin$ and $\fortinH$, we conclude
\begin{equation*}
\begin{aligned}
\|\Lambda_h u - \Lambda_H u\|^2
\leq{}
&
|\G_H(\Lambda_H u)-\G_h(\Lambda_h u)|_a \|\fortin z \|\\
&+
\| H (\G_H(\Lambda_H u) - \nabla_H (\Lambda_H u) ) \|_{\T_H\setminus\T_h}
\| H^{-1} (\fortin z - \fortinH z )\|
\\
\lesssim{}
&
\|\Lambda_h u - \Lambda_H u\|\\
&(|\G_H(\Lambda_H u)-\G_h(\Lambda_h u)|_a + \mu_H(\Lambda_H
u;\T_H\setminus\T_h)) .
\end{aligned}
\end{equation*}
\end{proof}

By passing to the limit in the statement of Proposition~\ref{p:drel}, and
observing that for $H$ small enough the second term on the right-hand side can
be absorbed, we obtain the standard reliability estimate.

\begin{corollary}[Reliability]\label{c:reliability}
Provided the initial mesh-size is sufficiently fine, we have
\begin{equation*}
\sum_{j\in J} d(u_j,\Lambda_h u_j)^2
\leq
\Crel^2
\sum_{j\in J} \mu_h(u_j,\T_h)^2 .
\end{equation*}
\end{corollary}


We conclude this section with a quasi-orthogonality property.

\begin{proposition}[Quasi-orthogonality]\label{l:qoI}
There exists a constant $\Cqo$ such that
\begin{equation*}
   d(\Lambda_h u, \Lambda_H u)^2
\leq 
d(u, \Lambda_H u)^2 - d( u, \Lambda_h u)^2
+
\Cqo \,\rho(h) 
 (d( u, \Lambda_h u)^2  + d(u, \Lambda_H u)^2 ).
\end{equation*}
\end{proposition}
\begin{proof}
The proof departs with the following obvious algebraic identities
\begin{equation*}
\begin{aligned}
 |\G_h(\Lambda_h u)-\G_H(\Lambda_H u)|_a^2
 &
 =
 |\sigma -\G_H(\Lambda_H u)|_a^2
 - |\sigma -\G_h(\Lambda_h u)|_a^2
\\
&\qquad
 - 2 a(\sigma - \G_h(\Lambda_h u),\G_h(\Lambda_h u) -\G_H(\Lambda_H u)) 
\end{aligned}
\end{equation*}
\begin{equation*}
 \|\Lambda_h u-\Lambda_H u\|^2
 =
 \|u -\Lambda_H u\|^2
 - \|u - \Lambda_h u \|^2
 - 2 (\Pi_h u - \Lambda_h u,\Lambda_h u -\Lambda_H u) .
\end{equation*}
The exact and discrete eigenvalue problems together with
the inclusion $\ddiv \Sigma_H \subseteq M_H$ imply
\begin{equation*}
\begin{aligned}
a(\sigma - \G_h (\Lambda_h u),\G_h (\Lambda_h u) - \G_H (\Lambda_H u)) 
&
=
-b(\G_h (\Lambda_h u) - \G_H (\Lambda_H u), u-\Lambda_h u)
\\
&
= \lambda (P^W_h u-P^W_H u, \Pi_h u-\Lambda_h u).
\end{aligned}
\end{equation*}

Therefore it follows from Proposition~\ref{p:hot}, Lemma~\ref{l:sigmanorm},
and the Young inequality that
\begin{equation*}
\begin{aligned}
&
\lvert a(\sigma - \G_h (\Lambda_h u),\G_h (\Lambda_h u) - \G_H (\Lambda_H u)) \rvert
+
\lvert (\Pi_h u - \Lambda_h u,\Lambda_h u - \Lambda_H u) \rvert
\\
&
\qquad
\leq 
\|\Pi_h u - \Lambda_h u \| \, 
  ( \|\Lambda_h u - \Lambda_H u\| + \lambda \|P^W_h u-P^W_H u\|)
\\
&
\qquad
\lesssim \rho(h) (d( u, \Lambda_h u)^2  + d(u, \Lambda_H u)^2 ).
\end{aligned}
\end{equation*}
\end{proof}

\subsection{Contraction property}

While the properties of the previous subsection are valid for any refinement
$\T_h$ of a mesh $\T_H$, in this section we deal with the mesh sequence
$\T_\ell$ which is the output of the adaptive strategy described in
Section~\ref{se:algorithm}.

The following property is quite standard and can be proved with the techniques
of~\cite{CasconKreuzerNochettoSiebert2008}.

\begin{lemma}[Error estimator reduction property]\label{l:reduction}
  Provided the initial mesh-size is sufficiently small such that
  the bulk criteria for $\mu_\ell$ and $\eta_\ell$ are equivalent (see
Lemma~\ref{l:estcomparison}),
  there exist constants $\rho_1\in(0,1)$ and $K\in (0,+\infty)$ such that
  $\T_\ell$ and its one-level refinement $\T_{\ell+1}$
  generated by AFEM satisfy
 \begin{equation*}
  \sum_{j\in J}\mu_{\ell+1}(u_j,\T_{\ell+1})^2
  \leq
  \rho_1 \sum_{j\in J}\mu_{\ell}(u_j,\T_\ell)^2
 +
 K_1 \sum_{j\in J} d(\Lambda_{\ell+1} u_j,\Lambda_{\ell} u_j)^2 .
 \end{equation*}
\end{lemma}

The following proposition presents the main contraction property which is
essential for the convergence proof of the adaptive strategy.

\begin{proposition}[Contraction property]\label{p:contraction}
 Provided the initial mesh-size is sufficiently small,
 there exist $\rho_2\in(0,1)$ and $\beta\in(0,+\infty)$ such that
 the term
\begin{equation}\label{e:xielldef}
  \xi_\ell^2=  \sum_{j\in J}\mu_\ell(u_j,\T_\ell)^2
                    + \beta \sum_{j\in J} d(u_j,\Lambda_\ell u_j)^2
\end{equation}
 satisfies
 \begin{equation*}
   \xi_{\ell+1}^2 \leq \rho_2 \xi_{\ell} ^2
        \quad \text{for all } \ell\in\mathbb{N}.
 \end{equation*}
\end{proposition}
\begin{proof}
Throughout the proof, we use the following notation
\[
  e^2_\ell = \sum_{j\in J} d(u_j,\Lambda_\ell u_j)^2\qquad
    \mu_{\ell}^2
        =\sum_{j\in J}\mu_\ell(u_j,\T_\ell)^2.
\]
The error estimator reduction from
Lemma~\ref{l:reduction} and the quasi-orthogonality
from Lemma~\ref{l:qoI} imply the following bound
\begin{equation*}
 \mu_{\ell+1}^2
  + K_1
     (1-\Cqo\rho(h_0))e_{\ell+1}^2
  \leq
  \rho_1 \mu_{\ell}^2
  +
  K_1 (1+\Cqo\rho(h_0))e_\ell^2 .
\end{equation*}
For any $\varepsilon\in (0,1)$, the last bound and
the reliability (Corollary~\ref{c:reliability}) give
 \begin{equation*}
\aligned
&  \mu_{\ell+1}^2
  + K_1
     (1-\Cqo\rho(h_0))e_{\ell+1}^2\\
&\qquad  \leq
  (\rho_1 + \varepsilon \Crel^2 K_1) \mu_{\ell}^2
  +
  K_1(1-\varepsilon+\Cqo\rho(h_0))e_\ell^2.
\endaligned
 \end{equation*}
 We take
 $\beta = K_1  (1-\Cqo\rho(h_0))$
 and
\begin{equation*}
\rho_2 = \max\left\{
 \rho_1 + \varepsilon \Crel^2 K_1 ,
 \frac{1-\varepsilon+\Cqo\rho(h_0)}{1-\Cqo\rho(h_0)} 
\right\},
\end{equation*}
so that
\begin{equation*}
\mu_{\ell+1}^2 + \beta e_{\ell+1}^2
\leq 
\rho_2 (\mu_{\ell}^2+ \beta e_\ell^2).
\end{equation*}
The choice of a sufficiently small $\varepsilon$ and
of a sufficiently small initial mesh-size $h_0$
leads to $\rho_2 < 1$.
\end{proof}

\section{Extension to three space dimensions}
\label{se:3D}

The results presented in the previous sections hold true also in three dimensions, 
provided the definitions of the computable and theoretical error indicators are 
modified as follows.

\begin{definition}
Let $\T_h$ be a simplicial decomposition of $\Omega$ and let
$(\sigma_{h,j},u_{h,j})\in\Sigma_h\times M_h$ be a discrete eigensolution
computed on the mesh $\T_h$.
Then, for all $T\in\T_h$ we define
\begin{equation*}
\eta_{h,j}(T)^2
=
   \| h_T (\sigma_{h,j} - \nabla u_{h,j}) \|_T^2 
   + \| h_T \ccurl \sigma_{h,j} \|_T^2 
   + \sum_{F\in\mathcal F(T)} h_F \| [\sigma_{h,j}]_F \times n_F\|_F^2,
\end{equation*}
where $h_T$ is the diameter of $T$, $\mathcal{F}(T)$ denotes the set of faces
of $T$, $h_F$ is the area of the face $F$, and $n_F$ is its unit normal 
vector. As usual, $[\sigma_h]_F\times n_F$ denotes the jump of the trace of
$\sigma_h\times n_F$ for internal faces and the trace for boundary faces.
\label{def:indicator3D}
\end{definition} 

\begin{definition}
Let $\T_h\in\mathbb{T}$ be a triangulation and let $(\sigma,u)$ 
be an eigensolution associated to the eigenvalue $\lambda$ (in
particular, this is used in the definition of $\Lambda_h$).
For all $T\in\T_h$ we define
\begin{equation*}
\aligned
\mu_h^2(u;T) 
={}&
   \| h_T (\G(\Lambda_h u) - \nabla \Lambda_h u) \|_T^2 
   + \| h_T \ccurl \G(\Lambda_h u) \|_T^2 \\
&   + \sum_{F\in\mathcal F(T)} h_F \| [\G(\Lambda_h u)]_F \times n_F\|_F^2 .
\endaligned
\end{equation*}
\end{definition}

In the three-dimensional case, the only proof which needs to be modified is the one of the discrete 
reliability of Proposition~\ref{p:drel} since it relies on the discrete Helmholtz decomposition which is 
different in two or three dimensions. 

Let $V_h$ denote the $H(\ccurl)$-conforming edge elements of  N\'ed\'elec 
(see~\cite{BoffiBrezziFortin2013}). 

Then, in the three dimensional case, the discrete Helmholtz decomposition reads (see~\cite{HuangXu2012}, Lemma 2.6): 
for any $\xi_h\in\Sigma_h$ there exist $\alpha_h\in M_h$ and $\mathbf{\beta}_h\in V_h$ such that 
$$
\xi_h=\G_h(\alpha_h)+\ccurl\beta_h.
$$

\begin{proposition}[Discrete reliability]

Provided the mesh-size of $\T_H$ is sufficiently small, we have
\begin{equation*}
\aligned
&|\G_h(\Lambda_h u)-\G_H(\Lambda_H u)|_a +
\| \Lambda_h (u) - \Lambda_H (u)\|\\
&\qquad\le\Cdrel
\mu_H (u;\tilde{\mathcal{R}})+C
\rho(H) (d(u,\Lambda_h u) + d(u,\Lambda_H u) ),
\endaligned
\end{equation*}
where $\tilde{\mathcal{R}}=\{T\in\T_H:\bar{T}\cap\bar{T'}\neq\emptyset\ \ \forall T'\in(\T_H\setminus\T_h)\}.$
\end{proposition}

\begin{proof}
Using the discrete Helmholtz decomposition, we write the error in the vectorial variable as 
\begin{equation*}
\G_h(\Lambda_h u) - \G_H(\Lambda_H u) 
     = \G_h(\alpha_h) + \ccurl \beta_h,
\end{equation*}
with $\alpha_h\in M_h$ and $\beta_h\in V_h$.

The term $|\G_h(\alpha_h)|_a$ can be treated without any modification as in the two dimensional case. 
Moreover, following the same argument as in~\cite[Lemma~3.1.]{HuangXu2012}, it can be proved that 
$$
|\ccurl \beta_h|_a \lesssim \mu_H(u;\tilde{\mathcal{R}}).
$$

As in the proof of Proposition~\ref{p:drel},
the error in the scalar variable can be bounded by using the 
duality argument of \cite{GaticaOyarzuaSayas2011,CarstensenPeterseimSchroeder2012}
and we can repeat the same arguments of the 2D case from Equation~\eqref{e:uguale}
onwards, concluding the proof. 
\end{proof}

\begin{remark}
Compared with the two-dimensional case,
in the three-dimensional version of the discrete reliability,
the set $\T_H\setminus\T_h$ is replaced with the slighliy
larger set $\tilde{\mathcal{R}}$ which essentially is $\T_H\setminus\T_h$
plus one additional layer of simplices.
The shape-regularity implies that there is a constant $C$ such that
\begin{equation*}
\card(\tilde{\mathcal{R}})\leq C \card(\T_H\setminus\T_h) .
\end{equation*}
and therefore the estimate \eqref{e:minimalMl} remains valid at the
expense of the multiplicative constant $C$,
and with this modification the proof of Theorem~\ref{th:main}
applies also to the three-dimensional case.
\end{remark}

\section*{Acknowledgements}
The second named author gratefully acknowledges the hospitality
of the Dipartimento di Matematica ``F. Casorati''
(University of Pavia) during his stay in September 2014.

\bibliographystyle{abbrv}
\bibliography{mixedEVPcluster}

\begin{thebibliography}{10}

\bibitem{Alonso1996}
A.~Alonso.
\newblock Error estimators for a mixed method.
\newblock {\em Numer. Math.}, 74(4):385--395, 1996.

\bibitem{DougActa}
D.~N. Arnold, R.~S. Falk, and R.~Winther.
\newblock Finite element exterior calculus, homological techniques, and
  applications.
\newblock {\em Acta Numer.}, 15:1--155, 2006.

\bibitem{BM}
R.~Becker and S.~Mao.
\newblock An optimally convergent adaptive mixed finite element method.
\newblock {\em Numer. Math.}, 111(1):35--54, 2008.

\bibitem{BeckerMaoShi2010}
R.~Becker, S.~Mao, and Z.~Shi.
\newblock A convergent nonconforming adaptive finite element method with
  quasi-optimal complexity.
\newblock {\em SIAM J. Numer. Anal.}, 47(6):4639--4659, 2010.

\bibitem{BDD}
P.~Binev, W.~Dahmen, and R.~DeVore.
\newblock Adaptive finite element methods with convergence rates.
\newblock {\em Numer. Math.}, 97(2):219--268, 2004.

\bibitem{Boffi2010}
D.~Boffi.
\newblock Finite element approximation of eigenvalue problems.
\newblock {\em Acta Numer.}, 19:1--120, 2010.

\bibitem{BoffiBrezziFortin2013}
D.~Boffi, F.~Brezzi, and M.~Fortin.
\newblock {\em Mixed finite element methods and applications}, volume~44 of
  {\em Springer Series in Computational Mathematics}.
\newblock Springer, Heidelberg, 2013.

\bibitem{bbg2}
D.~Boffi, F.~Brezzi, and L.~Gastaldi.
\newblock On the convergence of eigenvalues for mixed formulations.
\newblock {\em Ann. Scuola Norm. Sup. Pisa Cl. Sci. (4)}, 25(1-2):131--154
  (1998), 1997.
\newblock Dedicated to Ennio De Giorgi.

\bibitem{BraessVerfuerth1996}
D.~Braess and R.~Verf{\"u}rth.
\newblock A posteriori error estimators for the {R}aviart-{T}homas element.
\newblock {\em SIAM J. Numer. Anal.}, 33(6):2431--2444, 1996.

\bibitem{Carstensen1997}
C.~Carstensen.
\newblock A posteriori error estimate for the mixed finite element method.
\newblock {\em Math. Comp.}, 66(218):465--476, 1997.

\bibitem{CarstensenFeischlPagePraetorius2014}
C.~Carstensen, M.~Feischl, M.~Page, and D.~Praetorius.
\newblock Axioms of adaptivity.
\newblock {\em Comput. Math. Appl.}, 67(6):1195--1253, 2014.

\bibitem{CarstensenGallistlSchedensack2015}
C.~Carstensen, D.~Gallistl, and M.~Schedensack.
\newblock Adaptive nonconforming {C}rouzeix-{R}aviart {FEM} for eigenvalue
  problems.
\newblock {\em Math. Comp.}, 84(293):1061--1087, 2015.

\bibitem{CarstensenGedicke2011}
C.~Carstensen and J.~Gedicke.
\newblock An oscillation-free adaptive {FEM} for symmetric eigenvalue problems.
\newblock {\em Numer. Math.}, 118(3):401--427, 2011.

\bibitem{CarstensenGedicke2012}
C.~Carstensen and J.~Gedicke.
\newblock An adaptive finite element eigenvalue solver of quasi-optimal
  computational complexity.
\newblock {\em SIAM J. Numer. Anal.}, 50(3):1029--1057, 2012.

\bibitem{CarstensenPeterseimSchroeder2012}
C.~Carstensen, D.~Peterseim, and A.~Schr\"oder.
\newblock The norm of a discretized gradient in ${H}(div)^*$ for a posteriori
  finite element error analysis.
\newblock {\em Numer. Math.}, 2015.
\newblock To appear.

\bibitem{CarstensenRabus2011}
C.~Carstensen and H.~Rabus.
\newblock An optimal adaptive mixed finite element method.
\newblock {\em Math. Comp.}, 80(274):649--667, 2011.

\bibitem{CasconKreuzerNochettoSiebert2008}
J.~Cascon, C.~Kreuzer, R.~H. Nochetto, and K.~G. Siebert.
\newblock Quasi-optimal convergence rate for an adaptive finite element method.
\newblock {\em SIAM J. Numer. Anal.}, 46(5):2524--2550, 2008.

\bibitem{ChenHolstXu2009}
L.~Chen, M.~Holst, and J.~Xu.
\newblock Convergence and optimality of adaptive mixed finite element methods.
\newblock {\em Math. Comp.}, 78(265):35--53, 2009.

\bibitem{DaiHeZhou2014}
X.~Dai, L.~He, and A.~Zhou.
\newblock Convergence and quasi-optimal complexity of adaptive finite element
  computations for multiple eigenvalues.
\newblock {\em IMA J. Numer. Anal.}, 2014.
\newblock To appear, DOI 10.1093/imanum/dru059.

\bibitem{DaiXuZhou2008}
X.~Dai, J.~Xu, and A.~Zhou.
\newblock Convergence and optimal complexity of adaptive finite element
  eigenvalue computations.
\newblock {\em Numer. Math.}, 110(3):313--355, 2008.

\bibitem{dnr2}
J.~Descloux, N.~Nassif, and J.~Rappaz.
\newblock On spectral approximation. {II}. {E}rror estimates for the {G}alerkin
  method.
\newblock {\em RAIRO Anal. Num\'er.}, 12(2):113--119, iii, 1978.

\bibitem{Do}
W.~D{\"o}rfler.
\newblock A convergent adaptive algorithm for {P}oisson's equation.
\newblock {\em SIAM J. Numer. Anal.}, 33(3):1106--1124, 1996.

\bibitem{DR}
J.~Douglas, Jr. and J.~E. Roberts.
\newblock Mixed finite element methods for second order elliptic problems.
\newblock {\em Mat. Apl. Comput.}, 1(1):91--103, 1982.

\bibitem{DuranGastaldiPadra1999}
R.~G. Dur{\'a}n, L.~Gastaldi, and C.~Padra.
\newblock A posteriori error estimators for mixed approximations of eigenvalue
  problems.
\newblock {\em Math. Models Methods Appl. Sci.}, 9(8):1165--1178, 1999.

\bibitem{Gallistl2014thesis}
D.~Gallistl.
\newblock {\em Adaptive Finite Element Computation of Eigenvalues}.
\newblock Doctoral dissertation, Humboldt-Universit\"at zu Berlin,
  Mathematisch-Naturwissenschaftliche Fakult\"at II, Institut f\"ur Mathematik,
  2014.

\bibitem{Gallistl2014nonconf}
D.~Gallistl.
\newblock Adaptive nonconforming finite element approximation of eigenvalue
  clusters.
\newblock {\em Comput. Methods Appl. Math.}, 14(4):509--535, 2014.

\bibitem{Gallistl2014conform}
D.~Gallistl.
\newblock An optimal adaptive {FEM} for eigenvalue clusters.
\newblock {\em Numer. Math.}, 2014.
\newblock Published online, doi: 10.1007/s00211-014-0671-8.

\bibitem{GarauMorinZuppa2009}
E.~M. Garau, P.~Morin, and C.~Zuppa.
\newblock Convergence of adaptive finite element methods for eigenvalue
  problems.
\newblock {\em Math. Models Methods Appl. Sci.}, 19(5):721--747, 2009.

\bibitem{Gardini2009}
F.~Gardini.
\newblock Mixed approximation of eigenvalue problems: a superconvergence
  result.
\newblock {\em M2AN Math. Model. Numer. Anal.}, 43(5):853--865, 2009.

\bibitem{GaticaMaischak2005}
G.~N. Gatica and M.~Maischak.
\newblock A posteriori error estimates for the mixed finite element method with
  {L}agrange multipliers.
\newblock {\em Numer. Methods Partial Differential Equations}, 21(3):421--450,
  2005.

\bibitem{GaticaOyarzuaSayas2011}
G.~N. Gatica, R.~Oyarz{\'u}a, and F.-J. Sayas.
\newblock A residual-based a posteriori error estimator for a fully-mixed
  formulation of the {S}tokes-{D}arcy coupled problem.
\newblock {\em Comput. Methods Appl. Mech. Engrg.}, 200(21-22):1877--1891,
  2011.

\bibitem{GianiGraham2009}
S.~Giani and I.~G. Graham.
\newblock A convergent adaptive method for elliptic eigenvalue problems.
\newblock {\em SIAM J. Numer. Anal.}, 47(2):1067--1091, 2009.

\bibitem{GR}
V.~Girault and P.-A. Raviart.
\newblock {\em Finite element methods for {N}avier-{S}tokes equations},
  volume~5 of {\em Springer Series in Computational Mathematics}.
\newblock Springer-Verlag, Berlin, 1986.
\newblock Theory and algorithms.

\bibitem{HuangXu2012}
J.~Huang and Y.~Xu.
\newblock Convergence and complexity of arbitrary order adaptive mixed element
  methods for the {P}oisson equation.
\newblock {\em Sci. China Math.}, 55(5):1083--1098, 2012.

\bibitem{LarsonMalqvist2008}
M.~G. Larson and A.~M{\aa}lqvist.
\newblock A posteriori error estimates for mixed finite element approximations
  of elliptic problems.
\newblock {\em Numer. Math.}, 108(3):487--500, 2008.

\bibitem{LovadinaStenberg2006}
C.~Lovadina and R.~Stenberg.
\newblock Energy norm a posteriori error estimates for mixed finite element
  methods.
\newblock {\em Math. Comp.}, 75(256):1659--1674 (electronic), 2006.

\bibitem{NochettoSiebertVeeser2009}
R.~H. Nochetto, K.~G. Siebert, and A.~Veeser.
\newblock Theory of adaptive finite element methods: an introduction.
\newblock In {\em Multiscale, nonlinear and adaptive approximation}, pages
  409--542. Springer, Berlin, 2009.

\bibitem{NochettoVeeser}
R.~H. Nochetto and A.~Veeser.
\newblock Primer of adaptive finite element methods.
\newblock In {\em Multiscale and adaptivity: modeling, numerics and
  applications}, volume 2040 of {\em Lecture Notes in Math.}, pages 125--225.
  Springer, Heidelberg, 2012.

\bibitem{pasciak}
J.~E. Pasciak and J.~Zhao.
\newblock Overlapping {S}chwarz methods in {$H$}(curl) on polyhedral domains.
\newblock {\em J. Numer. Math.}, 10(3):221--234, 2002.

\bibitem{ScottZhang1990}
L.~R. Scott and S.~Zhang.
\newblock Finite element interpolation of nonsmooth functions satisfying
  boundary conditions.
\newblock {\em Math. Comp.}, 54(190):483--493, 1990.

\bibitem{Stevenson2007}
R.~Stevenson.
\newblock Optimality of a standard adaptive finite element method.
\newblock {\em Foundations of Computational Mathematics}, 7(2):245--269, 2007.

\bibitem{stevenson2008}
R.~Stevenson.
\newblock The completion of locally refined simplicial partitions created by
  bisection.
\newblock {\em Math. Comp.}, 77(261):227--241, 2008.

\bibitem{StrangFix1973}
G.~Strang and G.~J. Fix.
\newblock {\em An analysis of the finite element method}.
\newblock Prentice-Hall Series in Automatic Computation. Prentice-Hall, Inc.,
  Englewood Cliffs, N. J., 1973.

\bibitem{WohlmuthHoppe1999}
B.~I. Wohlmuth and R.~H.~W. Hoppe.
\newblock A comparison of a posteriori error estimators for mixed finite
  element discretizations by {R}aviart-{T}homas elements.
\newblock {\em Math. Comp.}, 68(228):1347--1378, 1999.

\end{thebibliography}

\end{document}